\documentclass[adraft]{eptcs}
\usepackage{aiml26}
\usepackage[noadjust]{cite}

\usepackage{iftex}
\ifpdf
  \usepackage{underscore}         % Only needed if you use pdflatex.            
  \usepackage[T1]{fontenc}        % Recommended with pdflatex                   
\else
  \usepackage{breakurl}           % Not needed if you use pdflatex only.        
\fi
\usepackage{bm}

\usepackage[a4paper, total={6in, 8in}]{geometry}
\usepackage{graphicx} % Required for inserting images
\let\oldalpha\alpha
\renewcommand{\alpha}{\scalebox{0.85}{$\oldalpha$}}

\usepackage{amsmath, amssymb}
\usepackage{centernot}
\usepackage{esvect}
\usepackage{semantic, syntax}
\usepackage{enumitem}
\usepackage{hyperref}
\usepackage{caption}
\usepackage{subcaption}
\usepackage{pgfplots}
\usepgfplotslibrary{fillbetween}
\pgfplotsset{compat=1.18}
\usepackage{multicol}

\usepackage{tikz}
\usetikzlibrary{3d,
    patterns,
    patterns.meta,
    arrows,
    arrows.meta,
    positioning,
    shapes,
    calc,
    intersections,
    hobby
}
\usepackage{bussproofs}

\usepackage[bibliography=common,appendix=inline]{apxproof}
\theoremstyle{plain}
\newtheoremrep{theorem}{Theorem}[section]
\newtheoremrep{proposition}[theorem]{Proposition}
\newtheoremrep{corollary}[theorem]{Corollary}
\newtheoremrep{lemma}[theorem]{Lemma}

\usepackage{totcount}
\newtotcounter{todonum}

\newcommand{\shelf}[1]{}
\newcommand{\change}[1]{#1}

\title{Some Results on Causal Modalities\\in General Spacetimes}
\author{Marco Lewis \qquad\qquad Nesta van der Schaaf
\institute{Université Paris-Saclay, CNRS, CentraleSupélec,\\ ENS Paris-Saclay, Inria, Laboratoire Méthodes Formelles,\\ 91190, Gif-sur-Yvette, France}}

\newcommand{\titlerunning}{Causal Modalities in General Spacetimes}
\newcommand{\authorrunning}{M. Lewis, N. van der Schaaf}
\hypersetup{
  bookmarksnumbered,
  pdftitle    = {\titlerunning},
  pdfauthor   = {\authorrunning},
  pdfsubject  = {EPTCS},               % Consider adding a more appropriate subject or description 
  pdfkeywords = {modal logic, spacetime}
}

\def\etc{\emph{etc}.}
\def\ie{\emph{i.e.},}
\def\eg{\emph{e.g.},}

\newcommand{\norm}[1]{\left\lVert#1\right\rVert}
\DeclareMathOperator{\im}{im}

\newcommand{\genericworld}{M}

\newcommand{\genericrelation}{\mathrel{\triangleleft}}
\newcommand{\ngenericrelation}{\ntriangleleft}
\newcommand{\genericeval}{V}
\newcommand{\atomicprops}{AP}
\newcommand{\mframe}[2]{\langle #1, #2 \rangle}
\newcommand{\genericframe}{\mframe{\genericworld}{\genericrelation}}
\newcommand{\mmodel}[3]{\langle #1, #2, #3 \rangle}
\newcommand{\genericmodel}{\mmodel{\genericworld}{\genericrelation}{\genericeval}}

\newcommand{\entail}[3]{#1\if\relax\detokenize{#2}\relax\else,\fi #2 \models #3}
\newcommand{\nentail}[3]{#1\if\relax\detokenize{#2}\relax\else,\fi #2 \not\models #3}

\newcommand{\mentail}[3]{\entail{#1}{#2}{#3}}
\newcommand{\nmentail}[3]{\nentail{#1}{#2}{#3}}
\newcommand{\mentailgeneric}[1]{\mentail{\genericmodel}{x}{#1}}

\newcommand{\chron}{\ll}
\newcommand{\nchron}{\not\chron}
\newcommand{\chroneq}{\mathrel{\underline{\chron}}}
\newcommand{\nchroneq}{\centernot\chroneq}
\newcommand{\horismos}{\rightarrow}
\newcommand{\caus}{\preccurlyeq}
\newcommand{\ncaus}{\centernot\preccurlyeq}
\newcommand{\after}{\mathrel{\alpha}}
\newcommand{\nafter}{\not\after}

\newcommand{\necc}{\square}
\newcommand{\poss}{\lozenge}

\newcommand{\inev}{\raise.4ex\hbox{$\bigtriangledown$}} % inevitably in the future (inevitably)
\newcommand{\modal}[1]{\mathcal{L}(#1)}

\newcommand{\afterformula}{a \alpha f}
\newcommand{\aftertwoformula}{a \alpha_2 f}
\newcommand{\mK}{\mathbf{K}}
\newcommand{\afterlogic}{\mathbf{D4da}}
\newcommand{\aftertwologic}{\mathbf{D4da_2}}
\newcommand{\gabb}{\text{(Gabb)}}

\tikzset{empty/.style={draw=none},
    refl/.style={fill=black},
    every label/.style={draw=none,inner sep=0},
}
\newcommand{\tikzminkowski}{\def\xmax{2}
          \draw[->,thick] (0,-\xmax) -- (0,\xmax+0.2) node[left=-.5] {$t$};
          \draw[->,thick] (-\xmax,0) -- (\xmax+0.2,0) node[below=0] {$x$};
}
 
\begin{document}

\maketitle

% \oldtodo{\total{todonum} TODOs left}
\begin{abstract}
    Causality is one of the fundamental structures of spacetimes, determining the possible behaviour and propagation of physical information. 
    Causal structure can be analysed through the various modal logics it induces.
    The modal logics for the chronological and causal relations of the archetypal Minkowski spacetime have been classified.
    However, only partial results have been achieved for the strict variant of the causal relation, known as the \emph{after relation}.
    \change{ 
    Towards classification, it was shown by Shapirovsky and Shehtman that the after modality in Minkowski space satisfies a formula we call the `after formula'.}

    \change{
    The present work continues this analysis towards arbitrary spacetimes.
    %Here, the different causal modalities are defined in terms of classes of curves obeying certain speed limits, which in turn define different Kripke frames.
    In particular, we prove that the after modality in any smooth spacetime satisfies the after formula.}
    % By utilizing the definition of the causal relations through causal paths, we can lift known results about the modal logics of Minkowski spacetime to general spacetimes.
    % In particular, for the after relation, we show that a previously studied formula within the logics of Minkowski spacetime holds in arbitrary spacetimes.
    %\change{By defining causal relations through causal paths, we uncover the sublogics for various frames based on spacetimes and their causal relations.
    %In particular, we show that the sublogic of the frame of a spacetime with the after relation includes a previously studied formula within the frame of Minkowski spacetime and the after relation.}
    We introduce a related modal formula that demonstrates that the logic of two-dimensional spacetimes are more expressive than higher-dimensional ones.
    Lastly, we study the interrelation between the logical properties and physical properties along the causal ladder.
    % a classification of causal structures according to a hierarchy of physically relevant properties.
\end{abstract}

\section{Introduction}
% Spacetimes and causal relation
% Chronological/causal relation at a high-level
\emph{Spacetime} provides the backdrop in which physical bodies and information propagate~\cite{penrose1972TechniquesDifferentialTopology}. An object can move through spacetime in one of two ways: 
\begin{itemize}
    \item \emph{Causally:} moving speeds up to the speed of light.
    The traditional interpretation is that all physical information must propagate in this way.
    If it is possible to get from $x$ to $y$ causally we write $x\caus y$;
    \item \emph{Chronologically:} moving strictly slower than the speed of light.
    All bodies with non-zero mass are understood to move this way.
    Similarly we then write $x\chron y$.
\end{itemize}

% Look at differences of spacetimes from logical perspective, need a language/logic to discuss this
From both a physical and philosophical standpoint it is of interest to understand why our universe exists in the way it does.
For example Tegmark~\cite{Tegmark1997} uses the theory of partial differential equations to argue spacetime must have one temporal dimension and three spatial dimensions, although the discussion is non-rigorous. 

Similarly, with the causal relations defined, as well as their variations~(Section~\ref{sec:spacetime:causalstruct}), it is possible to identify behaviours of the spacetime based on their mathematical properties.
For example, the existence of time travel (closed/looping causal curves) may be prohibited by assuming that $\caus$ is a \emph{partial order}.
Thus we can single out spacetimes with physically desirable properties.
The \emph{causal ladder} is a well-studied hierarchy of such properties of the causal structure of spacetimes, one of the cornerstones of mathematical relativity theory~\cite{Minguzzi08}.

In this work we approach the problem of identifying well-behaved spacetimes by taking a logical perspective.
This way physically desirable properties are related to logical formula or rules.
There is a broad body of literature on logics of time and computation~\cite{goldblatt1992LogicsTimeComputation}, with amongst the most well-known being temporal logic~\cite{DemriGorankoLange2016, GabbayHodkinsonReynolds1994}.
% Here, as a first step, we consider one of the simplest temporal logics: \emph{modal logic}~\cite{BoxesAndDiamonds,Blackburn2001}.
\change{Here, we consider the interactions between the frames of spacetimes and \emph{modal logic}~\cite{BoxesAndDiamonds,Blackburn2001}.}

% Modal logic history
There has been some research into the modal logic of \emph{Minkowski space}, the spacetime of special relativity (Example~\ref{example:minkowski space}).
One of the early works is by Goldblatt~\cite{Goldblatt1980}, who showed that the modal logic of any $1+n$-dimensional Minkowski spacetime with the causal relation~$\caus$ is~$\mathbf{S4.2}$.
Shehtman~\cite{Shehtman1983} also observed similar results around that time for various~$\mathbf{S4}$ logics.
Shapirovsky and Shehtman~\cite{Shapirovsky2002} showed that the chronological relation $\chron$, again for any $1+n$-dimensional Minkowski spacetime, is given by the modal logic $\mathbf{OI.2}$, which is a sublogic of $\mathbf{S4.2}$.
Recent works, such as \cite{Hirsch2018, Hirsch2022}, have been investigating the temporal modal logic, using both past and future modalities, of the strict variant of the causal relation (what is referred to as the \emph{after} relation; see Section~\ref{sec:spacetime:causalstruct}).
However, it has only been shown in two-dimensional Minkowski spacetime that the temporal logic of the after relation is decidable \cite{Hirsch2018}.
% It is still unknown what the actual modal/temporal logic of the after relation is, although previous works have noted it has unique properties~\cite{Goldblatt1980, Shapirovsky2005}.
\change{Whilst the frames of Minkowski spacetime with the chronological and causal relation are complete with respect to their logics, it is unknown what the modal logic of Minkowski spacetime with the after relation is and whether it is complete, although previous works have noted it has distinct properties~\cite{Goldblatt1980, Shapirovsky2005}.}

However, to the best of our knowledge, the modal logics of general spacetimes has not been investigated.\footnote{Though there are related approaches such as via domain theory~\cite{martin2006DomainSpacetimeIntervals} or orthomodular lattices~\cite{casini2002LogicCausallyClosed}.}
What is the modal logic of a general spacetime?
Some properties can be derived from the general differential geometric setup, such as that $\chron$ is transitive and 2-dense ($\mathbf{OI}$) and $\caus$ is a preorder ($\mathbf{S4}$).
One of the main difficulties in providing a more refined characterisation lies in the fact that, even though spacetimes locally behave like Minkowski space (locally inheriting the logic), general spacetimes can vary greatly in global topology and causal structure.
Thus, to get a better grasp on the situation in this work we focus on how the modal logics of the spacetime vary along where it sits on the causal ladder.
We demonstrate how as we add physically desirable properties (moving down the causal ladder), the respective modal logics change.

% Our work
\paragraph{Outline.}
The main goal of this work is to investigate modal logic within the context of the causal relations of general spacetimes.
We provide an introduction to spacetimes and modal logic in Sections~\ref{sec:spacetime} and~\ref{sec:modal} respectively.
The main results from this work are as follows:
\begin{enumerate}
    \item we define and discuss the properties of a new class of spacetimes on the causal ladder, referred to as \emph{causally non-totally vicious} spacetimes (Section~\ref{sec:spacetime:cntv});
    \item we find the sub-logics of different spacetimes based on their causal relations and note the importance of a formula introduce in~\cite{Shapirovsky2005} for spacetimes (Section~\ref{sec:modalspacetime});
    \item we show that the modal logic of two-dimensional spacetimes is larger than other $1+n$-dimensional spacetimes (Section~\ref{sec:separable}); and
    \item we find the sub-logics for spacetimes on several different levels of the causal ladder and show that there is a limit to the expressibility of causal ladder properties in modal logic (Section~\ref{sec:modalcausalladder}).
\end{enumerate}

\section{Spacetimes and the Causal Ladder}
\label{sec:spacetime}
In this section we recall the necessary definitions on \emph{smooth spacetimes} and outline the so-called \emph{causal ladder}. For general literature we refer to~\eg{}~\cite{penrose1972TechniquesDifferentialTopology, landsman2021FoundationsGeneralRelativity}, and for a review of the causal ladder to~\cite{minguzzi2019LorentzianCausalityTheory}.

% Old version
% Smooth spacetimes are differential geometric objects that model the unified physical notion of space and time, forming the mathematical foundation of Einstein's theory of general relativity. The main ingredient is a \emph{metric} $g$ defined on a \change{(real second countable Hausdorff)} smooth manifold~$M$, defining at each point $x\in M$ a nondegenerate symmetric bilinear form $g_x\colon T_x M \times T_x M\to \mathbb{R}$ on its tangent space, to be thought of as a generalised inner product. The difference is that $g_x$ need not be positive definite, meaning \change{non-zero vectors may have zero or even negative length.}
% \change{
% We say a metric~$g$ has \emph{Lorentzian signature}, if each~$g_x$ can be written as~$\mathrm{diag}(-1,1,\ldots,1)$ in some basis, which is interpreted as one \emph{time dimension} and $n$ \emph{spatial dimensions}.}
% \change{A \emph{global time orientation} is a smooth vector field $T$ determining a consistent notion of future.}
% \change{Putting it together, a \emph{spacetime} is a connected smooth manifold equipped with a Lorentzian metric and global time orientation.}

\change{
Formally, a \emph{spacetime} is a (real second countable Hausdorff connected) smooth manifold~$M$ equipped with a \emph{Lorentzian metric}~$g$ and a \emph{time orientation}~$T$. The metric defines at each point~$x\in M$ a nondegenerate symmetric bilinear form~$g_x\colon T_xM\times T_xM\to \mathbb{R}$ on the tangent space, to be thought of as a generalised inner product. The difference is that~$g_x$ need not be positive definite, meaning non-zero vectors may have zero or even negative length. The Lorentzian signature is interpreted as one \emph{time dimension} and $n$ \emph{spatial dimensions}. The time orientation $T$ is a global smooth vector field that determines a consistent notion of future.
}

%~~~~~~~~~~~~~~~~~~~~~~~~~~~~~~~~~~~~~~~~~~~~~~~~
\subsection{Causal Structure}
\label{sec:spacetime:causalstruct}
One of the most fundamental structures of spacetime is its \emph{causal structure}. This consists of two classes of curves, the \emph{causal curves} and \emph{chronological curves}. Causal curves describe the possible movement of bodies going at the speed of light or slower, and chronological curves describe paths going strictly slower than the speed of light. These curves determine fundamental restrictions of how information is allowed to propagate in the spacetimes. We follow the conventions of~\cite{minguzzi2019LorentzianCausalityTheory}. Fixing a time orientation $T$, a tangent vector $v\in T_xM$ at some point $x\in M$ is called:
\begin{multicols}{2}
    \begin{itemize}
        \item \emph{timelike} if $g_x(v,v)<0$;
        \item \emph{lightlike} if $v\neq 0$ and $g_x(v,v)=0$;
        \item \emph{causal} if it is timelike or lightlike, \ie{}~$g_x(v,v)\leq 0$ and $v\neq 0$.
    \end{itemize}
\columnbreak
    \begin{itemize}
        \item \emph{future directed} (fd) if $g_x(T_x,v)<0$;
        \item \emph{past directed} (pd) if $g_x(T_x,v)>0$.
    \end{itemize}
\end{multicols}
% A tangent vector $v\in T_xM$ at some point $x\in M$ in a spacetime is called:
% \begin{itemize}
%     \item \emph{timelike} if $g_x(v,v)<0$;
%     \item \emph{lightlike} if $v\neq 0$ and $g_x(v,v)=0$;
%     \item \emph{causal} if it is timelike or lightlike, i.e.~$g_x(v,v)\leq 0$ and $v\neq 0$.
% \end{itemize}
% If $T$ is a global time orientation on $M$, we say $v$ is:
% \begin{itemize}
%     \item \emph{future directed} (fd) if $g_x(T_x,v)<0$;
%     \item \emph{past directed} (pd) if $g_x(T_x,v)>0$.
% \end{itemize}
\noindent By \emph{curve} we mean a continuous (piecewise) smooth function $\gamma\colon I\to M$ defined on an interval~$I\subseteq \mathbb{R}$.
Denote $\im(\gamma) = \{\gamma(t):t\in I\}$.
We say a curve $\gamma$ is:
\begin{itemize}
    \item \emph{causal} if all velocity tangent vectors are fd causal;
    \item \emph{chronological} or \emph{timelike} if all velocity tangent vectors are fd timelike;
    \item \emph{lightlike} or a \emph{lightcurve} if all velocity tangent vectors are fd lightlike.
\end{itemize}

The fundamental \emph{causal relations} are defined in terms of these classes of curves. For $x,y\in M$ we define the \emph{causal(ity)}, \emph{chronology} and \emph{horismos relation} as follows, respectively:
\begin{itemize}
    \item $x\caus y$ iff there exists a causal curve from $x$ to $y$, or $x=y$;
    \item $x\chron y$ iff there exists a timelike curve from $x$ to $y$;
    \item $x\horismos y$ iff there exists a lightlike curve from $x$ to $y$, or $x=y$ (\ie{}~$x\caus y$ and not $x \chron y$).
\end{itemize}

These relations induce the \emph{causal} and \emph{chronological cones} ($J$ and $I$ respectively):
\begin{align*}
    J^{-}(x) = \{ z\in \genericworld : z\caus x \}, \qquad&\qquad J^{+}(x) = \{y\in \genericworld : x\caus y\},
    \\
    I^{-} (x) = \{ z\in \genericworld : z\chron x \}, \qquad&\qquad I^{+}(x) = \{y\in \genericworld : x\chron y\}.
\end{align*}

% \begin{itemize}
%     \item Chronological past: $I^{-}(x) = \{y \in \genericworld : y \chron x \}$;
%     \item Chronological future: $I^{+}(x) = \{y \in \genericworld : x \chron y \}$;
%     \item Causal past: $J^{-}(x) = \{y \in \genericworld : y \caus x \}$;
%     \item Causal future: $J^{+}(x) = \{y \in \genericworld : x \caus y \}$.
% \end{itemize}

Additionally, we define the \emph{after} (or \emph{strictly causal}) and \emph{reflexive chronological relation} as follows, respectively:
\begin{itemize}
    \item $x \after y$ iff there exists a causal curve from $x$ to $y$ \change{(read as ``$y$ happens after $x$'')};
    \item $x \chroneq y$ iff there exists a timelike curve from $x$ to $y$, or $x=y$.
\end{itemize}

The after relation was introduced in \cite{robb1914TheoryTimeSpace}. It is the `strict' version of the causal relation $\caus$, since causal curves are by definition non-degenerate, only allowing a reflexive pair $x\after x$ when there exists a causal loop from $x$ to $x$. Similarly, $\chroneq$ is the reflexive completion of $\chron$, thus always allowing $x \chroneq x$.
We write $\nchron$, $\ncaus$ , $\nafter$, or $\nchroneq$ if two points are not related by the respective relation. We say $x$ and $y$ are \emph{spacelike separated} if $x\ncaus y$ and $y\ncaus x$.

\begin{example}[Minkowski spacetime]\label{example:minkowski space}
One of the most important spacetimes is $1+n$-dimensional \emph{Minkowski space}, based on the manifold $\mathbb{R}^{1+n}$ (identified with its own tangent space at the origin) with metric defined via
\[
    \change{\eta}(x,y) = -x^0y^0 + \sum_{i=1}^n x^iy^i,
\]
where we denote generic points $x\in \mathbb{R}^{1+n}$ in components as $(x^0,\vec{x})=(x^0,x^1,\ldots,x^n)$.
%\begin{equation*}
%    g = - c^{2} dx_{n}^{2} + \sum_{i=0}^{n-1} dx_{i}^{2},
%\end{equation*}
The causal relations then have the following convenient characterisation in terms of the Euclidean norm~$\norm{\cdot}$ on~$\mathbb{R}^n$:
\begin{align}
    x \chron y 
    &\quad \text{iff}\quad
    y^0 > x^0 + \norm{\vec{y}-\vec{x}}, \label{eq:chron}
    \\
    x \caus y
    &\quad \text{iff}\quad
    y^0 \geq x^0 + \norm{\vec{y}-\vec{x}}, \label{eq:caus}
    \\
    x \horismos y
    &\quad \text{iff}\quad
    y^0 = x^0 + \norm{\vec{y}-\vec{x}}. \label{eq:horismos}
\end{align}

Figure~\ref{fig:mink2} depicts 2-dimensional Minkowski spacetime and the light lines from the origin.
These lines are at 45\textdegree~angles to the axes.
The highlighted interior region above (below) $t=0$ is the region where the future-directed (past-directed) chronological points are up to some $t'$.
Including the light lines on the border of the regions give all the points that are causal with respect to $(-1,0)$.
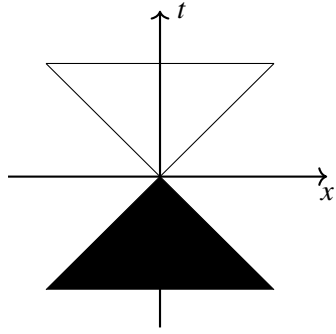
\begin{figure}[t]
    \centering
    \begin{tikzpicture}
        \def\x{0}
        \def\s{1.5}
      % \node[label={[label distance=-.1cm]-45:$0$}] (o) at (0,0) {};
        % \coordinate (fp) at (1.5,0);
        % \draw[fill=white] (fp) -- ([shift={(.25, .25)}]fp) -- ([shift={(-.25, .25)}]fp) -- (fp);
        % \draw[fill] (fp) -- ([shift={(.25,-.25)}]fp) -- ([shift={(-.25,-.25)}]fp) -- (fp);
        \fill[white] (\x,0) -- (\x+\s,\s) -- (\x-\s,\s);
        \fill[black] (\x,0) -- (\x+\s,-\s) -- (\x-\s,-\s);
        \draw (\x,0) -- (\x+\s,\s);
        \draw (\x,0) -- (\x-\s,\s);
        \draw (\x-\s,\s) -- (\x+\s,\s);
        \draw (\x,0) -- (\x+\s,-\s);
        \draw (\x,0) -- (\x-\s,-\s);
        \tikzminkowski;
    \end{tikzpicture}
    \caption{Two-dimensional Minkowski spacetime.}
    \label{fig:mink2}
\end{figure}

%We note that in $1+n$-dimensional Minkowski spacetime, the causal conditions can be defined in terms of its metric~\cite{Goldblatt1980, Shapirovsky2002}.
%For example, for two points $x,y \in \mathbb{R}^{1+n}$ we can say that $x\caus y$ iff $\sum_{i=0}^{n-1} (y_i - x_i)^2 \leq (y_{n} - x_{n})^2$ and $x_n \leq y_n$, which is the metric $g$ applied to the vector $v = y - x$ is required to be causal ($g_x(v,v) \leq 0$) and future directed ($g_x(T_x,v) < 0$).
%However, defining the causal conditions in terms of a spacetime's metric cannot be done in general.
\end{example}

%~~~~~~~~~~~~~~~~~~~~~~~
\subsection{Properties of the causal relations}
A central theme of this work is characterising the (modal) properties of the causal relations of a spacetime. We first collect some well-established results from the literature.

\begin{lemma}\label{lemma:relations between causal orders}
    We have the following relations between the causal orders:
    \begin{itemize}
        \item $x\chron y \implies x\after y \implies x\caus y$;
        \item if $x\neq y$ then $x\after y$ iff $x\chron y$ or $x\horismos y$;
    \end{itemize}
\end{lemma}

\begin{definition}
    A relation, ${\genericrelation} \subseteq S \times S$ is \emph{semi-full} if:
    \begin{itemize}
        \item (seriality) $\forall x \in S$, there exists $y \in S$ such that $x \genericrelation y$;
        \item (2-dense) if $x \genericrelation y_1, y_2$, there exists $z \in S$ such that $x \genericrelation z \genericrelation y_1, y_2$.
    \end{itemize}
\end{definition}

\begin{lemma}[\cite{penrose1972TechniquesDifferentialTopology, kronheimer1967StructureCausalSpaces}]
    We have the following relational properties:
    \begin{itemize}
        \item $\chron$ is semi-full, $\after$ is serial;
        \item $\chron$, $\chroneq$, $\after$, $\caus$ are transitive;
        \item $\caus$, $\chroneq,\horismos$ are reflexive.
    \end{itemize}
    \label{lem:relcommonproperties}
\end{lemma}

The following is one of the fundamental rules of causality theory in smooth spacetimes.
\begin{proposition}[Push-up Rule~\cite{penrose1972TechniquesDifferentialTopology}]
We have that for $x, y, z \in \genericworld$:
\begin{itemize}
    \item $x \chron y$, $y \caus z \implies x \chron z$;
    \item $x \caus y$, $y \chron z \implies x \chron z$.
\end{itemize}
\end{proposition}

Note that $\caus$ can be replaced with $\after$ in the push-up rule.

\begin{remark}
\label{rmk:diffafter}
    We make a few remarks on the history of the after relation within modal logic.
    Robb~\cite{robb1914TheoryTimeSpace} was first to define the after relation for spacetime.
    Goldblatt~\cite{Goldblatt1980} uses this to make some remarks on the modal logic of after in Minkowski spacetime.
    From here, two separate definitions of after have emerged (for $\mathbb{R}^n$).
    Shapirovsky and Shehtman~\cite{Shapirovsky2005}, Hirsch and McLean~\cite{Hirsch2018}/Brett~\cite{Hirsch2022}, and this work follow the definition using causal paths (referred to as \emph{causal after}, $\after$), whereas Phillips~\cite{Phillips1998}, and S. and J. Uckleman~\cite{Uckleman2007} define after (referred to as \emph{non-causal after}, $\after'$) as
    \begin{equation*}
        x\after'y~\quad\text{iff}\quad~\exists x_k < y_k~\text{and}~\forall j\neq k, x_j \leq y_j.
    \end{equation*}
    The latter definition does not capture the notion of causal paths that we are interested in but is still an interesting relation in its own right.
    The after relations are similar in $\mathbb{R}^2$: one relation is simply a rotation about the origin of the other, making them isomorphic and their modal logics the same.
    However, the relations have different behaviours in higher dimensions.
    The future cone of causal after is a cone whereas the one for non-causal after is a hypercube (Figure~\ref{fig:diffafter}), meaning their frames yield different logics.
    For instance, in 3-dimensions, $\after'$ has a property where at least two of any four distinct points (that are unrelated to each other) must share a common point after their root point~\cite{Phillips1998}, but $\after$ does not have this property~\cite{Shapirovsky2005} (one can pick four points such that there is no common point after the root).
    We do not study the modal logic of $\after'$ in this work, although some analysis we make in Section~\ref{sec:moreafter} and \ref{sec:separable:minkowski2} may apply to it.

    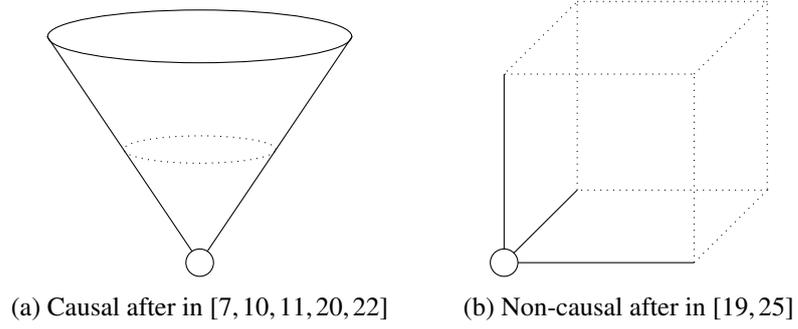
\begin{figure}[t!]
        \centering
        \begin{subfigure}{.35\textwidth}
            \centering
            \begin{tikzpicture}[scale=1.5]
                \node[ellipse,draw,minimum width=4cm, minimum height=.7cm,fill=white] (a) at (0,2) {};
                \node[ellipse,draw,dotted,minimum width=2cm, minimum height=.35cm,fill=white] (b) at (0,1) {};
                \draw (a.east) -- (0,0) -- (a.west);
                \node[draw,circle,fill=white] {};
            \end{tikzpicture}
            \caption{Causal after in~\cite{robb1914TheoryTimeSpace, Goldblatt1980, Shapirovsky2005, Hirsch2018, Hirsch2022}}
        \end{subfigure}
        \hspace{.5cm}
        \begin{subfigure}{.3\textwidth}
            \centering
            \begin{tikzpicture}[scale=2.5]
                \draw (1,0,0) -- (0,0,0);
                \draw (0,1,0) -- (0,0,0);
                \draw (0,0,-1) -- (0,0,0);
                \draw[dotted] (1,0,0) -- (1,1,0) -- (0,1,0);
                \draw[dotted] (1,0,0) -- (1,0,-1) -- (0,0,-1);
                \draw[dotted] (0,0,-1) -- (0,1,-1) -- (0,1,0);
                \draw[dotted] (0,1,-1) -- (1,1,-1) -- (1,0,-1);
                \draw[dotted] (1,1,0) -- (1,1,-1);
                \node[draw,circle,fill=white] {};
            \end{tikzpicture}
            \caption{Non-causal after in~\cite{Phillips1998, Uckleman2007}}
        \end{subfigure}
        \caption{After relations in $\mathbb{R}^3$.}
        \label{fig:diffafter}
    \end{figure}
\end{remark}

\subsection{The Causal Ladder}
\begin{figure}[t]
    \centering
    \begin{tikzpicture}[
        every path/.style={implies-, double equal sign distance}, 
        every node/.style={}, 
        node distance = 5mm and 5mm,
    ]
        \node[] (ntv) {Non-totally Vicious};
        
        \node[below = of ntv, minimum width = 3cm] (bntv) {};
        
        \node[below = of ntv] (chron) {Chronological};
        \node[below = of chron] (caus) {Causal};
        
        % \node[right = of bntv, minimum width = 3cm] (chron) {Chronological};
        % \node[left = of bntv, minimum width = 3cm] (cntv) {Causally Non-totally vicious};
        % \node[below = of bntv] (caus) {Causal};

        \node[below = of caus] (dist) {(Past-/Future-) Distinguishing};
        \node[below = of dist] (strongly) {Strongly Causal};
        \node[below = of strongly] (stably) {Stably Causal};
        \node[below = of stably] (gh) {Globally Hyperbolic};

        \draw (ntv) -- (chron);
        \draw (chron) -- (caus);
        
        % \draw (ntv.south east) -- (chron.north west);
        % \draw (ntv.south west) -- (cntv.north east);
        % \draw (chron.south west) -- (caus.east);
        % \draw (cntv.south east) -- (caus.west);
        
        \draw (caus) -- (dist);
        \draw (dist) -- (strongly);
        \draw (strongly) -- (stably);
        \draw (stably) -- (gh);
    \end{tikzpicture}
    \caption{The (Simple) Causal Ladder}
    \label{fig:causalladder}
\end{figure}
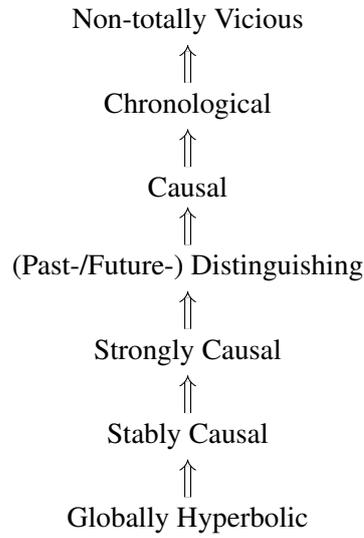

The causal ladder provides different classifications of spacetimes based on the properties of their causal relations. A basic version of it, sufficient for the present work, is depicted in Figure~\ref{fig:causalladder} (with the weakest property at the top).
Below we define the properties needed in this work (see Section~\ref{sec:modalcausalladder}); for more disambiguated versions see \eg{}~\cite[Figure 8]{minguzzi2019LorentzianCausalityTheory}.

\begin{definition}
    A spacetime is called:
    \begin{itemize}
        \item \emph{totally vicious} if $x\chron x$ for all $x\in M$ (\ie{}~$\chron$ is \emph{reflexive}).
        A spacetime is \emph{non-totally vicious} (NTV) if there is some $x \in M$ such that $x \nchron x$ (\ie{}~$\chron$ is not reflexive);
        
        \item \emph{chronological} if $x \nchron x$ for all $x \in M$ (\ie{}~$\chron$ is \emph{irreflexive});

        \item \emph{causal} if $x \caus y$ and $y \caus x$ imply $x = y$ (\ie{}~$\caus$ is \emph{anti-symmetric});

        \item \emph{past-distinguishing} (resp.~\emph{future-distinguishing}) if $I^{-}(x) = I^{-}(y)$ (resp.~$I^{+}(x) = I^{+}(y)$) implies $x = y$. If a spacetime is both past- and future-distinguishing it is called \emph{distinguishing};

        \item \emph{globally hyperbolic} if the \emph{causal diamonds} $J^{+}(x)\cap J^{-}(y)$ are compact in the manifold topology for all $x,y\in \genericworld$.
    \end{itemize}
\end{definition}

For proofs of implication between the different properties, see \cite{Minguzzi08,minguzzi2019LorentzianCausalityTheory}.
We note that the implications are strict, \eg{} there are chronological spacetimes that are not causal.
The study of the causal ladder and the underlying dependencies is one of the cornerstones of modern mathematical relativity theory.

We provide two examples of spacetime that reside on different levels of the causal ladder, but we will see other examples throughout the paper.

\begin{example}
    Minkowski space is globally hyperbolic. This follows straightforwardly from the description of the causal relation~$\caus$ in Equation~\eqref{eq:caus}, from which we see that the causal cones $J^{\pm}(x)$ are closed. Thus the diamonds $J^{+}(x)\cap J^{-}(y)$ are closed and bounded, and hence compact by the Heine-Borel theorem.
\end{example}
% \begin{example}
% Minkowski spacetime is globally hyperbolic.
% This involves taking the plane at certain time to act as the Cauchy surface.
% For example in 2-dimensional Minkowski spacetime, the line $t=0$ acts as a suitable Cauchy surface.
% \end{example}

\begin{example}\label{ex:Misner spacetime}
The spacetime depicted in Figure~\ref{fig:cntvchron:cntv} is a non-totally vicious spacetime.
The manifold is $\mathbb{R} \times S^1$, the surface of an infinitely long unit cylinder, and the lightcones are tilted from $90^\circ$ to being upright at $0^\circ$. This is known as the \emph{Misner spacetime}, see~\cite[Example 4.28]{minguzzi2019LorentzianCausalityTheory} for details. It is non-totally vicious because points on the dotted line, $(0, \theta)$, cannot chronologically reach themselves (even though causally they can).
\end{example}

\subsection{Causally Non-Totally Vicious: A New Property for the Causal Ladder}
\label{sec:spacetime:cntv}
Causal ladder properties can have multiple and equivalent definitions.
In particular, the causal step has an interesting equivalent condition.

\begin{lemma}
    Let $\genericworld$ be a spacetime.
    Then $\genericworld$ is causal, \ie{}~$\caus$ is anti-symmetric, iff $\after$ is irreflexive.
\end{lemma}
% \begin{proof}
%     $\Rightarrow$ : Assume $\genericworld$ is a causal spacetime, and suppose that there exists $x \in \genericworld$ such that $x \after x$.
%     Thus, there is some causal curve $\gamma : [a,b] \to \genericworld$ with $\gamma(a) = \gamma(b) = x$ and $\gamma(t) \neq x$ for some $t \in [a,b]$ (since the underlying tangent vector $v \neq 0$).
%     By taking $x' = \gamma(t)$, we have that $x \after x' \after x$.
%     Therefore, $x \caus x'$ and $x' \caus x$.
%     Since $\genericworld$ is causal, then we must have that $x' = x$, but this is a contradiction as $x' \neq x$.
%     Thus, $\after$ must be irreflexive.

%     $\Leftarrow$ : Now assume that $\after$ is irreflexive.
%     Let $x, y \in \genericworld$ such that $x \caus y$ and $y \caus x$.
%     Assume $x \neq y$, then it must be the case that $x \after y$ and $y \after x$.
%     As $\after$ is transitive, we have that $x \after x$, which is a contradiction as $\after$ is irreflexive.
%     Therefore, $x = y$ and thus $\genericworld$ is causal.
% \end{proof}

\begin{remark}
    A similar equivalence can be made between chronological spacetimes, where $\chron$ is irreflexive, and spacetimes where $\chroneq$ is anti-symmetric.
    In fact, this correspondence ($\triangleleft$ is irreflexive iff $\trianglelefteq$ is anti-symmetric) holds for any transitive relation~$\triangleleft$ with the property that $x \triangleleft x$ iff $x \triangleleft x' \triangleleft x$ for some $x' \neq x$, and where $\trianglelefteq$ is the reflexive completion of $\triangleleft$.
\end{remark}

Both the chronological and causal properties on the causal ladder are affecting the spacetime in a similar way, but affect different relations.
With the chronological relation, $\chron$, there is a difference between losing reflexivity (being non-totally vicious) and gaining irreflexivity (being chronological).
However, the causal property in the current ladder combines the loss of reflexivity and gaining of irreflexivity on $\alpha$ into one ladder property.
By changing the perspective of the causal property affecting $\after$, we can introduce a class of spacetimes that separates the notion of losing reflexivity and gaining irreflexivity.
Thus, we get a causal analogue of non-total vicious.

\begin{definition}
A spacetime is \emph{causally non-totally vicious (cNTV)} iff $\exists x \in \genericworld$ such that $x \nafter x$.
\end{definition}

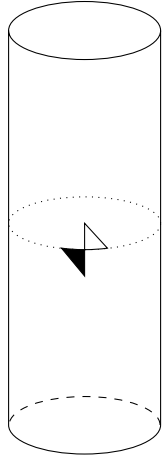
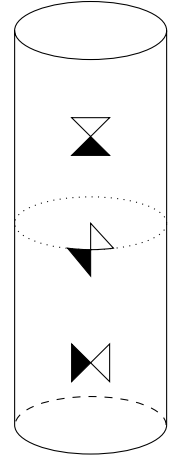
\begin{figure}[t]
    \centering
    \begin{subfigure}[t]{.45\textwidth}
        \centering
        \begin{tikzpicture}
    \node[cylinder, draw, shape border rotate = 90, aspect=3, minimum height = 6cm, minimum width=2cm] (A) {};
    % \draw[dashed] (M.after bottom) -- (M.before bottom);
    \draw[dashed]
    let \p1 = ($ (A.after bottom) - (A.before bottom) $),
        \n1 = {0.5*veclen(\x1,\y1)-\pgflinewidth},
        \p2 = ($ (A.bottom) - (A.after bottom)!.5!(A.before bottom) $),
        \n2 = {veclen(\x2,\y2)-\pgflinewidth}
  in
    ([xshift=-\pgflinewidth] A.before bottom) arc [start angle=0, end angle=180,
    x radius=\n1, y radius=\n2];

    \node[ellipse,draw,dotted,minimum width=2cm, minimum height=.7cm] at (0,.35) {};
    \newcommand{\lightconedistance}{.35}
    \coordinate (p) {};
    \draw[fill=white] (p) arc (-90:-60:.6 and .15) -- (0,\lightconedistance) -- (p);
    \draw[fill] (p) arc (-90:-120.5:.6 and .15) -- (0,-\lightconedistance) -- (p);
    
\end{tikzpicture}
        \caption{A chronological (non-causal) spacetime that is not cNTV.
        Every point acts lightlike on a slice of the cylinder allowing every point to (causally) loop around to itself, but a point cannot reach itself in a timelike way.}
        \label{fig:cntvchron:chron}
    \end{subfigure}
    \hfill
    \begin{subfigure}[t]{.45\textwidth}
        \centering
        \begin{tikzpicture}
% \node[rectangle,draw,thick,minimum size=5cm] (S) {};

% \draw[latex-latex] (S.south west) -| ([shift={(-.5,0)}]S.south west) --  ([shift={(-.5,0)}]S.north west) node[midway, left] {Identity} |- (S.north west);
% \draw[latex-latex] (S.south west) |- ([shift={(0,-.5)}]S.south west) --  ([shift={(0,-.5)}]S.south east) node[midway, below] {Identity} -| (S.south east);

% \newcommand{\killingfieldwidth}{0.33*5cm}
% \newcommand{\lightconelength}{0.35cm}
% \node[rectangle,draw,minimum height = 5cm, minimum width=\killingfieldwidth,fill=black!10] (K) {};

% \coordinate (l) at (-2.5, 0) {};
% \coordinate (r) at (2.5, 0) {};
% \coordinate (lk) at (-.5*\killingfieldwidth, 0) {};
% \coordinate (rk) at (.5*\killingfieldwidth, 0) {};

% \draw[fill=white] (l) -- (-2.75,.25) -- (-2.25,.25) -- (l);
% \draw[fill] (l) -- (-2.75,-.25) -- (-2.25,-.25) -- (l);

% \draw[fill=white] (r) -- (2.75,.25) -- (2.25,.25) -- (r);
% \draw[fill] (r) -- (2.75,-.25) -- (2.25,-.25) -- (r);

% \draw (-2.5, 0) -- (2.5,0);

% \draw[fill=white] (lk) -- (-.5*\killingfieldwidth,\lightconelength) -- (-.5*\killingfieldwidth+\lightconelength,0) -- (lk);
% \draw[fill] (lk) -- (-.5*\killingfieldwidth,-\lightconelength) -- (-.5*\killingfieldwidth-\lightconelength,0) -- (lk);

% \draw[fill=white] (rk) -- (.5*\killingfieldwidth,\lightconelength) -- (.5*\killingfieldwidth-\lightconelength,0) -- (rk);
% \draw[fill] (rk) -- (.5*\killingfieldwidth,-\lightconelength) -- (.5*\killingfieldwidth+\lightconelength,0) -- (rk);

\node[cylinder, draw, shape border rotate = 90, aspect=3, minimum height = 6cm, minimum width=2cm] (A) {};
\draw[dashed]
let \p1 = ($ (A.after bottom) - (A.before bottom) $),
    \n1 = {0.5*veclen(\x1,\y1)-\pgflinewidth},
    \p2 = ($ (A.bottom) - (A.after bottom)!.5!(A.before bottom) $),
    \n2 = {veclen(\x2,\y2)-\pgflinewidth}
in
([xshift=-\pgflinewidth] A.before bottom) arc [start angle=0, end angle=180,
x radius=\n1, y radius=\n2];

\node[ellipse,draw,dotted,minimum width=2cm, minimum height=.7cm] at (0,.35) {};

\newcommand{\lightconedistance}{.35}
\coordinate (p) {};
\draw[fill=white] (p) arc (-90:-60:.6 and .15) -- (0,\lightconedistance) -- (p);
\draw[fill] (p) arc (-90:-120.5:.6 and .15) -- (0,-\lightconedistance) -- (p);
% \draw[fill=white] (p) -- ([shift={(0, \lightconedistance)}]p) -- ([shift={(\lightconedistance, 0)}]p) -- (p);
% \draw[fill] (p) -- ([shift={(0, -\lightconedistance)}]p) -- ([shift={(-\lightconedistance, 0)}]p) -- (p);

\coordinate (fp) at (0,1.5);
\coordinate (pp) at (0,-1.5);

\draw[fill=white] (fp) -- ([shift={(.25, .25)}]fp) -- ([shift={(-.25, .25)}]fp) -- (fp);
\draw[fill] (fp) -- ([shift={(.25,-.25)}]fp) -- ([shift={(-.25,-.25)}]fp) -- (fp);

\draw[fill=white] (pp) -- ([shift={(.25, .25)}]pp) -- ([shift={(.25,-.25)}]pp) -- (pp);
\draw[fill] (pp) -- ([shift={(-.25, .25)}]pp) -- ([shift={(-.25,-.25)}]pp) -- (pp);
\end{tikzpicture}
        \caption{A cNTV (non-causal) spacetime that is not chronological~\cite[Fig. 10]{minguzzi2019LorentzianCausalityTheory}.
        The points above the boundary are irreflexive in $\alpha$ but the points below the boundary are reflexive in $\after$ and $\chron$.}
        \label{fig:cntvchron:cntv}
    \end{subfigure}
        \caption{Spacetimes for Lemma~\ref{lem:cntvnotchron}.}
    \label{fig:cntvchron}
\end{figure}

The cNTV property makes $\alpha$ lose reflexivity and now the causal property only gives $\alpha$ the irreflexive property, separating the relational properties into individual steps on the causal ladder.
There is an obvious question to ask: where does cNTV sit on the causal ladder?
\begin{proposition}
    A causal spacetime is cNTV, %(causal $\Rightarrow$ cNTV)
    and a cNTV spacetime is NTV. %(cNTV $\Rightarrow$ NTV).
    \label{prop:causcntvntv}
\end{proposition}
The proof of these implications is simply based on the definitions of the properties and is obvious. Note further that the spacetimes in Figure~\ref{fig:cntvchron} are counterexamples to the converse: Figure~\ref{fig:cntvchron:chron} is NTV but not cNTV, and Figure~\ref{fig:cntvchron:cntv} is cNTV but not causal.

\subsubsection{Chronological Spacetimes and cNTV}
Given the result in Proposition~\ref{prop:causcntvntv}, the next question to ask is: what is the relation between a spacetime being cNTV and it being chronological?
\begin{theoremrep}
    We have that
    \begin{enumerate}[label=(\roman*)]
        \item there are chronological spacetimes that are not cNTV (Figure~\ref{fig:cntvchron:chron});
        \item there are cNTV spacetimes that are not chronological (Figure~\ref{fig:cntvchron:cntv}).
    \end{enumerate}
    \label{lem:cntvnotchron}
\end{theoremrep}
\begin{proof}
The proof of both can be observed in the spacetimes depicted in Figure~\ref{fig:cntvchron}, where each has one property but not the other. Both spacetimes use the same manifold, $\mathbb{R} \times S^{1}$. The metric of Figure~\ref{fig:cntvchron:chron} is just quotiented Minkowski space, and Figure~\ref{fig:cntvchron:cntv} is Misner spacetime from Example~\ref{ex:Misner spacetime}.
\end{proof}
% \begin{proof}
% The proof of both can be observed in the spacetimes depicted in Figure~\ref{fig:cntvchron}, where each has one property but not the other. Both spacetimes use the same manifold, $\mathbb{R} \times S^{1}$. The metric of Figure~\ref{fig:cntvchron:chron} is \todo{explain $\otimes$}
% \begin{equation}
%     g = - \dd t \otimes \dd \theta
%     % g = - \cos{\dd t} \otimes \cos{\dd \theta}?
%     \label{eq:chronmetric}
% \end{equation}
% for every point on the spacetime.
% The metric for Figure~\ref{fig:cntvchron:cntv}, provided in \cite[Example 4.28]{minguzzi2019LorentzianCausalityTheory}, is
% \begin{equation}
%     \begin{aligned}
%         & g = - \alpha \otimes \beta,
%         & & \alpha = -\sin{f(t)} \dd \theta + \cos{f(t)} \dd t,
%         & & \beta = \cos{f(t)} \dd \theta + \sin{f(t)} \dd t;
%     \end{aligned}
%     \label{eq:cntvmetric}
% \end{equation}
% where $f(t) \in (0, \pi/2]$ for $t > 0$ and $f(t) \in [-\pi/2, 0)$ for $t < 0$.
% The metric in Equation~\ref{eq:chronmetric} is an instance of the metric in Equation~\ref{eq:cntvmetric} by ignoring the constraints on $f$ and setting $f(t) = 0$ for all $t \in \mathbb{R}$.
% \end{proof}

The spacetime given in Figure~\ref{fig:cntvchron:cntv} was shown to be NTV and not chronological in \cite{minguzzi2019LorentzianCausalityTheory}, but here we can see that it is also cNTV.

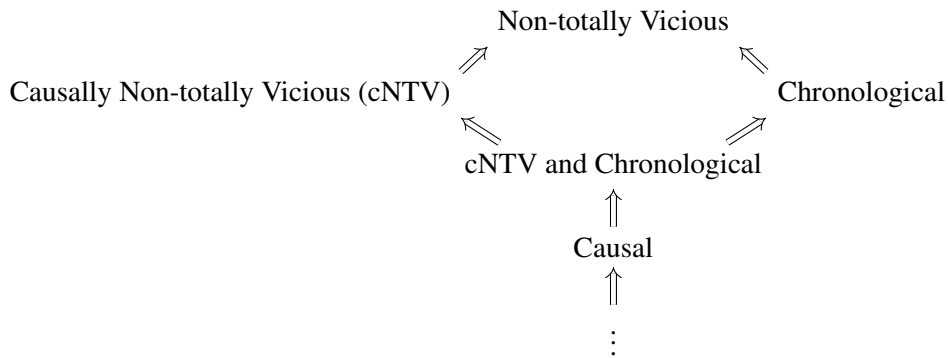
\begin{figure}[t]
    \centering
    \begin{tikzpicture}[
        every path/.style={implies-, double equal sign distance}, 
        every node/.style={}, 
        node distance = 5mm and 5mm,
    ]
        \node[] (ntv) {Non-totally Vicious};
        \node[right = of bntv, minimum width = 2.5cm] (chron) {Chronological};
        \node[left = of bntv, minimum width = 2.5cm] (cntv) {Causally Non-totally Vicious (cNTV)};
        \node[below = of bntv] (cntvchron) {cNTV and Chronological};
        \node[below = of cntvchron] (caus) {Causal};
        \node[below = of caus] (dist) {$\vdots$};
        
        \draw (ntv.south east) -- (chron.north west);
        \draw (ntv.south west) -- (cntv.north east);
        \draw (chron.south west) -- ([shift={(1.5,0)}]cntvchron.north);
        \draw (cntv.south east) -- ([shift={(-1.5,0)}]cntvchron.north);
        \draw (cntvchron) -- (caus);
        \draw (caus) -- (dist);
    \end{tikzpicture}
    \caption{New version of the NTV end of the causal ladder. Again, all implications are strict.}
    \label{fig:newcausalladder}
\end{figure}

\begin{remark}
\label{rmk:cntvchron}
It is clear that a causal spacetime is cNTV and chronological, but we further have that a spacetime that is cNTV and chronological is not necessarily causal.
A counterexample can be observed by taking the spacetime in Figure~\ref{fig:cntvchron:chron} and removing a point from it.
All points on the same circular slice as the removed point will be irreflexive in $\alpha$ since the hole prevents any causal path continuously circling around, but other points will remain reflexive.
The spacetime will remain chronological and become cNTV, but not be causal.
\end{remark}

The cNTV property changes one end of the causal ladder to look different to the original ladder given in Figure~\ref{fig:causalladder}.
The new end of the causal ladder is depicted in Figure~\ref{fig:newcausalladder}.
In Section~\ref{sec:modalcausalladder}, we will see how introducing the cNTV property into the causal ladder will make a clearer distinction of the modal logic of a spacetime to its position on the causal ladder.

\shelf{comment on topological/spacetime variant of cNTV/chronological?}

\subsubsection{Reflecting Spacetimes and cNTV}
We make a brief aside on another interesting property that affects a spacetime's position on the causal ladder.
\begin{definition}
    A spacetime $M$ is called \emph{reflecting} if for all $x,y\in M$ we have $I^{+}(x) \supseteq I^{+}(y)$ if and only if $I^{-}(x) \subseteq I^{-}(y)$.
\end{definition}

These have been studied in \cite{hawking1974CausallyContinuousSpacetimes, clarke1988ReflectingSpacetimes} and there are many equivalent conditions for a spacetime to be reflecting.
As concluded in \cite{clarke1988ReflectingSpacetimes}, the reflecting condition is ``only slightly weaker than global hyperbolicity''.
In \cite[Proposition 4.26]{minguzzi2019LorentzianCausalityTheory} it is proven that any reflecting NTV spacetime has to be chronological.
The following shows that the causal analogue of this statement does not hold.
\begin{lemmarep}
    There are reflecting cNTV spacetimes that are not causal.
\end{lemmarep}
\begin{proof}
    Consider the spacetime from Figure~\ref{fig:cntvchron:chron}.
    We use \cite[Theorem~2.1]{clarke1988ReflectingSpacetimes}, which states that removing regions with sufficiently small Hausdorff dimension ($< 3$) from a reflecting spacetime does not affect the reflecting property.
    The spacetime in Figure~\ref{fig:cntvchron:chron} is clearly reflecting, and by \cite[Theorem~2.1]{clarke1988ReflectingSpacetimes} we obtain another reflecting spacetime by removing a single point (points have a Hausdorff dimension of 0).
    The resulting spacetime is cNTV, as explained in Remark~\ref{rmk:cntvchron}, yet remains non-causal.
\end{proof}

While the cNTV property is unaffected by a spacetime being reflecting, a causal analogue of the reflecting property may affect it.
We leave it to future work to study more intricate relations between cNTV and the rest of the causal ladder.

\section{Modal Logic}
\label{sec:modal}
Now we introduce modal logic, which includes the standard logical connectives (conjuction, disjunction, \etc), atomic propositions, and modal operators ($\necc, \poss$).
Throughout, we assume a finite set of atomic propositions, $\atomicprops = \{p_1, p_2, \dots\}$.
We refer to \cite{Blackburn2001,BoxesAndDiamonds} for common definitions and theorems.

% Unlike the assumption of a linear past that is common in model checking, we assume that a point may have multiple possible pasts rather than a single historical past.
% This allows us to take into account the branching nature of spacetime and be able to reason about properties that hold in the past and future.
% Ockhamist past = linear, single history

% Kripke frames and Models
We begin by introducing the notion of a Kripke frame.
\begin{definition}
    A \emph{(Kripke) frame} is a tuple $\genericframe$, consisting of
    \begin{itemize}
        \item a set of points/worlds, $\genericworld$;
        \item a relation between points, ${\genericrelation} \subseteq \genericworld \times \genericworld$.
    \end{itemize}
\end{definition}

We may then also attach an evaluation function to relate atomic propositions to points within a frame to make a model.

\begin{definition}
    A \emph{model} is a triple $\genericmodel$, consisting of a frame $\genericframe$ together with an \emph{evaluation function} $\genericeval: \atomicprops \to 2^{\genericworld}$ that assigns atomic propositions to the points that they hold in.
\end{definition}

For a relation, $\genericrelation$, let $\genericrelation^{-1} = \{ (y,x) : (x, y) \in~\genericrelation\}$.
For $x \in \genericworld$, let $\genericrelation(x) = \{y \in \genericworld : x\genericrelation y\}$; and for $S \subseteq \genericworld$, let $\genericrelation(S) = \bigcup_{x \in S} \genericrelation(x)$.

\subsection{Grammar}
Standard modal logic is defined by the grammar
\begin{equation*}
\begin{aligned}
    \phi ::=~& \top \mid p \in \atomicprops \mid \phi_1 \lor \phi_2 \mid \lnot \phi \mid \necc p.
\end{aligned}
\end{equation*}
The grammar consists of truth ($\top$), logical or ($\lor$), logical not ($\lnot$) and the ``necessary'' operator ($\necc$).
Other logical connectives are defined in the usual way: falsity ($\bot$), logical and ($\land$), implication ($\Rightarrow$), equivalence ($\Leftrightarrow$).
Additionally, the ``possibility'' operation ($\poss$) is the dual of $\necc$:
\begin{equation*}
    \lnot \necc \lnot \phi \equiv \poss \phi.
\end{equation*}

The semantics of modal logic on a model, $\genericmodel$, at a point, $x \in \genericworld$, is defined recursively on the terms of the modal formula:
\begin{equation*}
    \begin{aligned}
        & \mentailgeneric{\top},~\forall x \in \genericworld;
        \\ & \mentailgeneric{p~(\in \atomicprops)} \text{ iff } x \in \genericeval(p);
        \\ & \mentailgeneric{\phi_1 \lor \phi_2} \text{ iff } \mentailgeneric{\phi_1} \text{ or } \mentailgeneric{\phi_2};
        \\ & \mentailgeneric{\lnot\phi} \text{ iff not } \mentailgeneric{\phi};
        \\ & \mentailgeneric{\necc \phi} \text{ iff } \forall y \in \genericworld \text{ such that } x \genericrelation y, \mentail{\genericmodel}{y}{\phi}.
    \end{aligned}
\end{equation*}

By removing one of the components of the model or the point from the left hand side of the semantic relation means the property holds for all possible instances of the component.
For example, ${\mentail{\genericmodel}{}{\phi}}$ means $\forall x \in \genericworld$, $\mentailgeneric{\phi}$ and we write $\mentail{\genericframe}{}{\phi}$ when $\phi$ holds on all evaluation functions and points.
Throughout, we write $|[\phi|]_{\genericmodel}$ to denote the set of points that entail a property $\phi$ in a model $\genericmodel$, \ie{} $|[\phi|]_{\genericmodel} = \{ x \in \genericworld : \mentailgeneric{\phi} \}$.

% Logics
\subsection{Normal Modal Logics}
The Necessitation rule is
\begin{prooftree}
\AxiomC{A}
\LeftLabel{(Nec)}
\RightLabel{,}
\UnaryInfC{$\necc$A}
\end{prooftree}
and the simplest axiom that holds in all modal logics is the modal formula
\begin{equation*}
    K := \necc (A \rightarrow B) \rightarrow (\necc A \rightarrow \necc B).
\end{equation*}
Other common modal formula are defined along with their corresponding properties of frames:
\begin{equation*}
    \begin{aligned}
        & a4 := \poss\poss A \rightarrow \poss A, & \text{transitivity}; \\
        & aT := A \rightarrow \poss A, & \text{reflexivity}; \\
        & aD := \poss \top, & \text{seriality}; \\
        & ad := \poss A \rightarrow \poss\poss A, & \text{density}; \\
        & ad_2 := \poss A \land \poss B \rightarrow \poss(\poss A \land \poss B), & \text{2-density}; \\
        & a2 := \poss\necc A \rightarrow \necc \poss A, & \text{confluence.\footnotemark} \\
    \end{aligned}
\end{equation*}
\footnotetext{Confluence is also referred to as the Church-Rosser property, the diamond property, or being weakly directed.}

The smallest modal logic is $\mK$, which contains the standard propositional logic axioms, \emph{modus ponens}, $K$, and (Nec).
A normal modal logic is the smallest modal logic that consists of $\mK$; any extra (valid) formulae $a_1, a_2 \dots$; and any formula that can be obtained from applying the rules in $\mK$ on the formulae given.
The respective logic is denoted $\mK + a_1 + a_2 +\dots$~.

We denote some normal modal logics derived from the formulas given above:
\begin{equation*}
    \begin{aligned}
        & \mathbf{K4} := \mK + a4, & \mathbf{T} := \mK + aT, \\
        & \mathbf{D4} := \mathbf{K4}+ aD, & \mathbf{S4} := \mathbf{K4} + aT, \\
        & \mathbf{OI} := \mathbf{D4} + ad_2, & \bm{\Lambda.2} := \bm{\Lambda} + a2;
    \end{aligned}
\end{equation*}
where $\bm{\Lambda}$ is an arbitrary normal modal logic.
Each of these has a combination of the properties of the corresponding modal formula.

As standard, for a set of formulae, $\Gamma$, and a formula, $\phi$, let $\Gamma \vdash \phi$ mean that $\phi$ can be logically derived from $\Gamma$ ($\phi \in \mK + \Gamma$) and $\Gamma \nvdash \phi$ iff $\Gamma \vdash \lnot \phi$.

For a frame, $F = \genericframe$, we write $\modal{F}$ to mean the set of formula that are satisfied under all possible evaluations of $F$.
This set of formula could be some logic, $\bm{\Lambda}$, and we may write $\modal{F} = \bm{\Lambda}$.
% Additionally, we may use other standard set notation between logics and formulas of frames.

\subsection{Bisimulation}
\label{sec:bisim}
\change{We recall the notion of bisimulation, needed within Section~\ref{sec:modalcausalladder}.}

\begin{definition}
Let $\mathcal{M} = \genericmodel$ and $\mathcal{M}' = \mmodel{\genericworld'}{\genericrelation'}{\genericeval'}$ be models.
A binary relation $\varnothing \subset Z \subseteq \genericworld \times \genericworld'$ is a bisimulation between $\mathcal{M}$ and $\mathcal{M}'$ if
\begin{enumerate}[label=(\roman*)]
    \item $w\genericrelation v$ and $w Z w'$, then there exists $v' \in \genericworld'$ such that $v Z v'$ and $w' \genericrelation' v'$;
    \item $w' \genericrelation v'$ and $w Z w'$, then there exists $v \in \genericworld$ such that $v Z v'$ and $w \genericrelation v$;
    \item $w \in \genericeval(p)$ and $wZw'$, then $w' \in \genericeval'(p)$.
\end{enumerate}
\end{definition}

\begin{theorem}
Modal formulae are invariant under bisimulation.
\end{theorem}

\subsection{Irreflexivity}
A final topic we introduce within modal logic is Gabbay's Irreflexive Rule~\cite{Gabbay1981}.
It is well known that there is no modal formula that can characterise irreflexive frames ($\forall x, \lnot x \genericrelation x$).
However, Gabbay found that there is a rule that exists, which when added to a normal modal logic characterises the irreflexive frames of that class (except with reflexive frames).
The rule is
\begin{prooftree}
\AxiomC{$\lnot (p \rightarrow \poss p) \rightarrow \phi$}
\LeftLabel{$\gabb$}
\RightLabel{if $p$ does not occur in $\phi$.}
\UnaryInfC{$\phi$}
\end{prooftree}
For a normal modal logic $\bm{\Lambda}$ without reflexivity ($aT$), let $\bm{\Lambda} + \gabb$ be the smallest modal logic of $\bm{\Lambda}$ with the addition of being closed under $\gabb$.
A frame whose logic is $\bm{\Lambda} + \gabb$ is irreflexive.

\section{Spacetime and Modal Logic}
\label{sec:modalspacetime}
We begin by looking at modal logic within the context of a generic spacetime.
No assumptions are made on where the spacetime lies on the causal ladder.
We analyse the common sets in spacetimes, as well as the normal modal logics that act as a subset for the modal logic of spacetimes.

\subsection{Past and Future}
We make some remarks on the modal operators and what properties they represent.
\begin{lemma}
$|[\poss \phi|]_{\mmodel{\genericworld}{\chron}{\genericeval}} = I^{-}(|[\phi|]_{\mmodel{\genericworld}{\chron}{\genericeval}})$
and
$|[\poss \phi|]_{\mmodel{\genericworld}{\caus}{\genericeval}} = J^{-}(|[\phi|]_{\mmodel{\genericworld}{\caus}{\genericeval}})$
\end{lemma}
\begin{proof}
    Obvious from the semantics of $\poss$ on $\chron (\caus)$ and the definition of $I^{-}~(J^{-})$.
\end{proof}

\begin{corollary}
$|[\necc \phi|]_{\mmodel{\genericworld}{\chron}{\genericeval}} = M \backslash I^{-}(M \backslash |[\phi|]_{\mmodel{\genericworld}{\chron}{\genericeval}})$ (similarly for $\caus, J^{-}$ respectively).
\end{corollary}

% Essentially, $\poss \phi$ generates the past light cone of whatever set is obtained from $\phi$ (those points that can possibly reach $\phi$) and $\necc \phi$ is all points whose future light cones are entirely within $\phi$ (those points who will necessarily be in $\phi$ no matter what path they take).

\subsection{Logic of Spacetime}

Given the definitions of the various causal relations, the relational properties they have (known from Lemma~\ref{lem:relcommonproperties}), and the properties on frames that normal modal logics describe; we have the following results on all spacetimes.
\begin{corollary}
    Let $\genericworld$ be a spacetime. Then, we have that
    \begin{enumerate}[label=(\roman*)]
        \item $\mathbf{S4} \subseteq \modal{\mframe{\genericworld}{\chroneq})}, \modal{\mframe{\genericworld}{\caus}}$;
        \item $\mathbf{OI} \subseteq \modal{\mframe{\genericworld}{\chron}}$;
        \item $\mathbf{D4} + ad \subseteq \modal{\mframe{\genericworld}{\after}}$.
    \end{enumerate}
    \label{cor:spacetimeclass}
\end{corollary}

This can be seen simply from the fact that the normal modal logics yield frames that have the associated properties and the frames of the spacetime with the causal relations have those properties.
For instance, $\mathbf{S4}$ yields frames that are transitive and reflexive, and $\caus$ is a transitive and reflexive relation.

\subsubsection{The After Formula}
In \cite{Shapirovsky2005}, it was shown that the formula, which we call the \emph{after formula},
\begin{equation*}
    \afterformula := \poss ( \poss (p_1 \land \lnot p_2 \land \necc \lnot p_2)
                \land \poss (p_2 \land \lnot p_1 \land \necc \lnot p_1))
                \land \poss q
                \Rightarrow
                \poss (\poss p_1 \land \poss q) \lor \poss (\poss p_2 \land \poss q),
\end{equation*}
holds in the modal logic of any $1+n$-dimensional Minkowski spacetime with the $\after$ relation.
This formula corresponds to the first-order property:
\begin{equation}
\begin{aligned}
    \forall x, y, y_1, y_2, z:
    &\Big[ \left(x \genericrelation y \land y \genericrelation y_1 \land y \genericrelation y_2 \land x \genericrelation z
    \land y_1 \neq y_2 \land \lnot(y_1  \genericrelation  y_2) \land \lnot(y_2  \genericrelation  y_1) \right)
    \\ &\rightarrow
    \exists t \left(x \genericrelation t \land t \genericrelation z \land (t \genericrelation y_1 \lor t \genericrelation y_2 ) \right) \Big].
    \footnotemark
\end{aligned}
\end{equation}
\footnotetext{Note $y_1 \neq y_2 \land \lnot(y_1  \genericrelation  y_2) \land \lnot(y_2  \genericrelation  y_1) \equiv y_1 \genericrelation^{\bowtie} y_2$ in \cite{Shapirovsky2005}.}
Observe that $\afterformula$ is a modified version of 3-density: if two (unrelated) points are related by a common point from the root and there is a third point that can be reached from the root, then at least one of the two points and the third point must share a common point from the root.
In fact, $\afterformula$ is a Sahlqvist formula (see \cite{Blackburn2001}), which provides some properties.

We make the following observation of $\afterformula$ within the scope of normal modal logics:
\begin{lemmarep}
    We have
    \begin{enumerate}[label=(\roman*)]
        \item $\mathbf{D4.2} + ad \nvdash \afterformula$; \label{lem:aaf:d4.2d}
        \item $\mathbf{T} \nvdash \afterformula$; \label{lem:aaf:t}
        \item $\mathbf{K4} + ad_2 \vdash \afterformula$. \label{lem:aaf:k4d2}
    \end{enumerate}
    \label{lem:afterformula}
\end{lemmarep}
\begin{proof}
Claims~\ref{lem:aaf:d4.2d} and \ref{lem:aaf:t} are proven by models that act as counterexamples given in Figures~\ref{fig:afnotin:d4d} and \ref{fig:afnotin:T} respectively.

For~\ref{lem:aaf:k4d2}, let $\mathcal{M} = \genericmodel$ be a $\mathbf{K4} + ad_2$ model and let $x \in \genericworld$.
If we have that
\begin{equation*}
\mentail{\mathcal{M}}{x}{\poss ( \poss (p_1 \land \lnot p_2 \land \necc \lnot p_2) \land \poss (p_2 \land \lnot p_1 \land \necc \lnot p_1)) \land \poss q};
\end{equation*}
then $\mentail{\mathcal{M}}{x}{\poss p_1}$ (by transitivity and properties of $\land$).
Since we have $\mentail{\mathcal{M}}{x}{\poss p_1 \land \poss q}$, then by 2-density we must have that $\mentail{\mathcal{M}}{x}{\poss(\poss p_1 \land \poss q)}$ and therefore $\mentail{\mathcal{M}}{x}{\afterformula}$.
Since $x, \mathcal{M}$ are generic, then $\mathbf{K4} + ad_2 \vdash \afterformula$.
\end{proof}

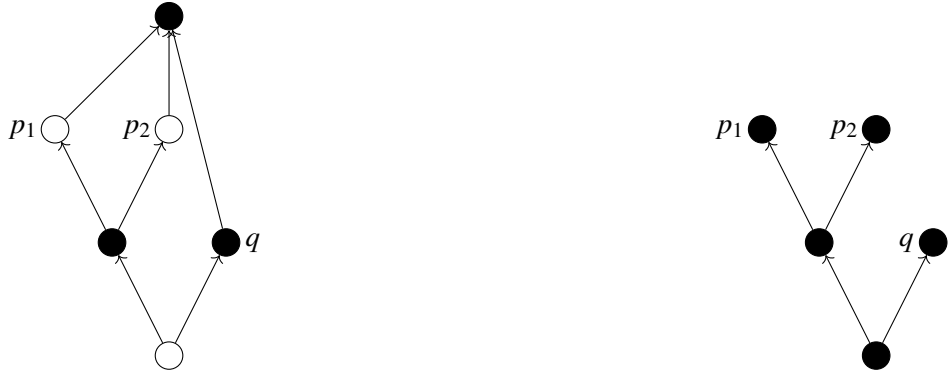
\begin{figure}[t]
    \centering
    \begin{subfigure}{.4\textwidth}
    \centering
    \begin{tikzpicture}[nodes={draw, circle}, ->,]
    \node (s) {} [grow'=up]
        child {node[refl] (0) {}
            child {node (00) [label=left:{$p_1$}] {}}
            child {node (01) [label=left:{$p_2$}] {}
                child {node[refl] (T) {}}
            }
        }
        child {node[refl] (1) [label=right:{$q$}] {}
        };

    \draw (00) -- (T);
    \draw (1) -- (T);
    \end{tikzpicture}
    \caption{A frame that is transitive, serial, dense, and confluent ($\mathbf{D4.2} + ad$).}
    \label{fig:afnotin:d4d}
    \end{subfigure}
    \hfill
    \begin{subfigure}{.4\textwidth}
    \centering
    \begin{tikzpicture}[nodes={draw, circle}, ->,]
    \node[refl] (s) {} [grow'=up]
        child {node[refl] (0) {}
            child {node[refl] (00) [label=left:{$p_1$}] {}}
            child {node[refl] (01) [label=left:{$p_2$}] {}}
        }
        child {node[refl] (1) [label=left:{$q$}] {}};
    \end{tikzpicture}
    \caption{A reflexive frame where the relation is non-transitive ($\mathbf{T}$).}
    \label{fig:afnotin:T}
    \end{subfigure}
    \caption{Counter models to $\afterformula$ in specific normal modal logics.}
    \label{fig:afnotin}
\end{figure}

Note that reflexivity implies 2-density.
Additionally, since $\mathbf{D4.2} + ad \nvdash \afterformula$, then any sub-logics of $\mathbf{D4.2} + ad$ ($\mathbf{D4}$, $\mathbf{K4}$, \dots) do not contain $\afterformula$.
Conversely, any logics that contain $\mathbf{K4} + ad_2$ as a sub-logic, such as $\mathbf{S4}$, contain $\afterformula$.
Therefore, since $\afterformula$ holds in transitive and 2-dense (or reflexive) frames, then it holds in frames that use a spacetime and one of $\chron$, $\chroneq$ or $\caus$.

Let $\afterlogic := \mathbf{D4} + ad + \afterformula$.\footnote{In \cite{Shapirovsky2005}, the logic $\mathbf{L\alpha_0} = \afterlogic\mathbf{.2} = \mathbf{D4.2} + ad + \afterformula$ is used for an analysis on Minkowski spacetime.}
Since $ad$ and $\afterformula$ are Sahlqvist formula, then by the Sahlqvist Theorem~\cite{Blackburn2001} canonicty is achieved for the formulas and $\afterlogic$.
\begin{corollary}
    The logic $\afterlogic$ is canonical.
\end{corollary}

%~~~~~~~~~~~~~~~~~~~~~~~~~~~~~~~~~~~~~~~~~~~~~~~~~
\subsubsection{Intermezzo: Geodesic Normal Coordinates}
\noindent
In this intermezzo we define and collect some technical results from causality theory that are needed to prove the after formula holds in general spacetimes. We refer to \cite[\S 5]{landsman2021FoundationsGeneralRelativity} for details.

Given any subset $A\subseteq M$, we write $x\chron_A y$ if there exists a fd timelike curve from $x$ to $y$ that lies entirely within $A$. Clearly $x\chron_A y$ implies~$x\chron y$. Define~$\caus_A$ and~$\after_A$ similarly.

A curve $\gamma$ is called a \emph{geodesic} if it is length extremising (technically: $\nabla_{\dot\gamma}\dot\gamma=0$ with respect to the Levi-Civita connection), which is interpreted as the motion of a non-accelerating, free falling body. For the present work it is important to note that geodesics are locally uniquely determined by their starting position and initial velocity. This ensures the existence of the \emph{exponential map}: for $x\in M$ it is the smooth map
$\exp_x\colon \mathcal V_x\subseteq T_xM\to M$ defined on a suitable open neighbourhood $\mathcal V_x$ of $0$ by $\exp_x(v):=\gamma_v(1)$, where $\gamma_v$ is the unique geodesic with $\gamma_v(0)=x$ and $\dot\gamma_v(0)=v$. The region $\mathcal{V}_x$ can be shrunk to a convex open neighbourhood of $0$ such that $\exp_x$ becomes a diffeomorphism onto its image $\exp_x(\mathcal{U}_x)=U_x$, which is called a \emph{(convex) normal neighbourhood} of $x$. This diffeomorphism facilitates what is known as \emph{geodesic normal coordinates}. Under these coordinates, geodesics within the neighbourhood $U_x$ emanating from $x$ correspond precisely to straight line segments in $T_xM$ through the origin.

\begin{theorem}\label{thm:GNC}
    Every $x\in M$ admits an open convex normal neighbourhood $U_x\subseteq M$ together with a diffeomorphism~${\Phi:=\exp_x^{-1}\colon U_x\to \mathcal{U}_x}$ onto an open convex neighbourhood $\mathcal{U}_x$ of the origin in Minkowski space,
    such that $\Phi(x)=0$, affinely parametrised geodesics $\gamma$ in $U_x$ with $\gamma(0)=x$ are transformed into straight lines as~$\Phi\circ \gamma (t)= t\dot{\gamma}(0)$, and for every $y\in U_x$ we have:
    \begin{equation}
    \label{eq:local relations}\tag{$\star$}
        x \chron_{U_x} y \iff 0 \chron \Phi(y),
        \qquad
        x \caus_{U_x} y \iff 0 \caus \Phi(y),
        \qquad
        x \after_{U_x} y \iff 0 \after \Phi(y).
    \end{equation}
\end{theorem}
\begin{proof}
    See~\cite[Theorem~5.5]{landsman2021FoundationsGeneralRelativity} and \cite[Corollary~2.10]{minguzzi2019LorentzianCausalityTheory}.
\end{proof}

\begin{remark}
    \change{
    Of course, the dual of Equation~\eqref{eq:local relations} also holds, where $y$ appears in the past of $x$.}
    \change{
    Note however that $\Phi$ does not preserve the causal relations between arbitrary points in $U_x$: only to and from the basepoint~$x$.}
    Beware similarly that local causal relations in a spacetime, as described by the Minkowskian causal structure via $\Phi$, do not necessarily translate to global causal relations.
    For example, consider the cylinder and point shown in Figure~\ref{fig:cntvchron:chron}.
    Moving vertically up the cylinder appears locally to be strictly lightlike with respect to the starting point, but globally one can wind round the cylinder to reach the end point chronologically.
\end{remark}

%~~~~~~~~~~~~~~~~~~~~~~~~~~~~~~~~~~~~~~~~~~~~~~~~~
\subsubsection{After Formula in Spacetime}
\label{sec:spacetime:afterformula}
As mentioned previously, it was shown in~\cite{Shapirovsky2005} that the after formula is satisfied in the after modality of any Minkowski space.

\begin{proposition}[\cite{Shapirovsky2005}]\label{thm:minkowskiafter}
    Consider $1+n$-dimensional Minkowski space $\mathbb{R}^{1+n}$ with its after relation~$\after$. Then
    \begin{equation*}
        \entail{\mframe{\mathbb{R}^{1+n}}{\after}}{}{a\alpha f}.
    \end{equation*}
\end{proposition}

\begin{remark}
    Following on from Remark~\ref{rmk:diffafter}, $\afterformula$ does not hold in $\mathbb{R}^n$ with non-causal after ($\after'$) for $n \geq 3$.
    This is because one can pick two points on a face of a hypercube and another point on an orthogonal axis such that the least common point is the root point.
\end{remark}

% \begin{theorem}\label{thm:GNC}
%     Every $x\in M$ admits an open neighbourhood $U_x\subseteq M$ together with a diffeomorphism~${\Phi\colon U_x\to \mathcal{U}_x}$ onto an open (convex) neighbourhood $\mathcal{U}_x$ of the origin in Minkowski space, such that the timelike/lightlike/causal curves within $U_x$ from $x$ are precisely the images under~$\Phi^{-1}$ of timelike/lightlike/causal curves in Minkowski space from the origin.
%     In particular, $\Phi$ and its inverse preserve $\caus$, $\chron$ and $\after$.
% \end{theorem}

Now we generalise this result, showing that $\afterformula$ holds in \emph{any} spacetime.
\change{The main technical step towards proving the after formula is satisfied in any spacetime is to prove the following property: if $x\after y \after y_1,y_2$ and $y_1\neq y_2$ are spacelike separated, then $x\chron y_i$ for some~$i$. First we give an elementary proof in Minkowski space, which is then lifted to arbitrary spacetimes using geodesic normal coordinates.}

\begin{lemmarep}
    In Minkowski space, let $x\horismos y \horismos y_1,y_2$ and $x\horismos y_1,y_2$. Then $y_1\horismos y_2$ or $y_2\horismos y_1$.
    \label{lem:mink:horismos}
\end{lemmarep}
\begin{proof}
    Take some arbitrary $z$ in the intersection of the lightcones of $x$ and $y$. Then we get from Equation~\eqref{eq:horismos} that
    $$
    \norm{\vec{z}-\vec{x}}= z^0- x^0 
    \qquad\text{and}\qquad 
    \norm{\vec{z}-\vec{y}}=z^0-y^0,
    $$
    and hence
    $$
    \norm{\vec{z}-\vec{x}}-\norm{\vec{z}-\vec{y}} = y^0 -x^0 = \norm{\vec{y}-\vec{x}}.
    $$
    This shows that equality is attained in the triangle inequality:
    $$
    \norm{\vec{z}-\vec{x}}
    \leq 
    \norm{\vec{z}-\vec{y}} + \norm{\vec{y}-\vec{x}},
    $$
    which occurs precisely when there exists a scalar $\lambda\geq 0$ with $\vec{z}-\vec{y}=\lambda(\vec{y}-\vec{x})$.
    In particular, for our situation, we get $\lambda_1,\lambda_2\geq 0$ with $\vec{y_i} = \vec{y}+ \lambda_i(\vec{y}-\vec{x})$. Note now by $y\horismos y_i$ that
    $$
    y_i^0 = y^0 + \norm{\vec{y_i}-\vec{y}}
    =
    y^0 + \lambda_i\norm{\vec{y}-\vec{x}} = y^0 + \lambda_i(y^0-x^0),
    $$
    so in particular $y_2^0 - y_1^0 = (\lambda_2-\lambda_1)(y^0-x^0)$. Similarly note $\vec{y_2}-\vec{y_1} = (\lambda_2-\lambda_1)(\vec{y}-\vec{x})$, and so we get
    $$
    \norm{\vec{y_2}-\vec{y_1}} = |\lambda_2-\lambda_1|\norm{\vec{y}-\vec{x}}
    =
    |\lambda_2-\lambda_1|(y^0-x^0)
    = 
    |y_2^0 - y_1^0|,
    $$
    which is precisely what it means for $y_1\horismos y_2$ or $y_2\horismos y_1$ to hold.
\end{proof}

\begin{corollary}\label{corollary:unique lightline in Minkowski}
    In Minkowski space, if $x\after y\after y_1,y_2$ and $y_1\neq y_2$ are spacelike separated, then~${x\chron y_i}$ for some $i$.
\end{corollary}
\begin{proof}
    Transitivity implies $x\after y_1,y_2$. So either $x\chron y_i$ for some $i$, in which case we are done, or $x\horismos y_1,y_2$. Moreover, if $x\chron y$ or $y\chron y_i$ for some $i$ we get the desired relation via the push-up rule. We are therefore left to consider the case that $x\horismos y\horismos y_1,y_2$. However, by Lemma~\ref{lem:mink:horismos}, this would imply $y_1\horismos y_2$ or $y_2\horismos y_1$, contradicting that $y_1$ and $y_2$ are spacelike separated.
\end{proof}

\change{The next few lemmas generalise this result to arbitrary spacetimes. We need some more definitions. A subset $A\subseteq M$ is called \emph{achronal} if $I^{+}(A)\cap A=\varnothing$, and a curve $\gamma$ is called \emph{achronal} if $\im(\gamma)$ is achronal. Explicitly, this means there are no times $s,t$ so that~$\gamma(s)\chron \gamma(t)$. A geodesic from~$p$ to~$q$ is called \emph{maximising} if it has greater or equal length than every other causal curve from~$p$ to~$q$. We then have the powerful result~\cite[Theorem~2.22]{minguzzi2019LorentzianCausalityTheory}, saying achronal lightlike geodesics are precisely the maximising lightlike ones.}

\change{
\begin{theoremrep}
\label{thm:2.22}
    Let $\gamma$ be a causal curve connecting $p$ to $q\neq p$. Either: there is a timelike curve~$\sigma$ from $p$ to $q$ whose length is strictly greater than that of $\gamma$; or $\gamma$ is a maximising geodesic (up to parametrisation).
    In particular, if there are no timelike curves connecting $p$ to $q$ then $\gamma$ is an achronal lightlike geodesic (up to parametrisation).    
\end{theoremrep}}

\change{
\begin{lemmarep}\label{lem:corner rounding}
Let $U\subseteq M$ be open, and take $p,q\in U$. If there exists a causal curve~$\gamma$ in $U$ from~$p$ to~$q\neq p$ that is not an achronal lightlike geodesic, then $p\chron_U q$.
\end{lemmarep}
\begin{proof}
Consider the spacetime $(U,g|_U)$. Applying Theorem~\ref{thm:2.22} to the given curve $\gamma$ in $U$ from~$p$ to~$q\neq p$, either we get a timelike curve $\sigma$ facilitating $p\chron_U q$, in which case we are done, or $\gamma$ is a maximising geodesic. If $\gamma$ is itself timelike we are also done, but if $\gamma$ is a maximising lightlike geodesic then it is achronal, contradicting the hypothesis.
\end{proof}
}

\change{
\begin{lemmarep}\label{lem:lift chronology}
Let $x\in M$, and take geodesic normal coordinates $\Phi\colon U_x\to \mathcal{U}_x$ from Theorem~\ref{thm:GNC}.
Suppose $p,q\in U_x$ satisfy~${p\after_{U_x} x \after_{U_x} q}$. If $\Phi(p)\chron \Phi(q)$ in Minkowski space, then $p\chron_{U_x} q$.
\end{lemmarep}

\begin{proof}
If $p\chron_{U_x} x$ or $x\chron_{U_x} q$ then the desired result follows from the push-up rule. Suppose therefore that $p\horismos_{U_x} x$ and $x\horismos_{U_x} q$. Let $\alpha$ and $\beta$ be the curves facilitating these relations, and let $\gamma$ be their concatenation. We claim that $\gamma$ is not a lightlike geodesic, so in particular not an achronal lightlike geodesic. From that, the desired~$p\chron_{U_x}q$ will follow by~Lemma~\ref{lem:corner rounding}.

Suppose for the sake of contradiction that $\gamma$ is a lightlike geodesic, affinely parametrised so that~$\gamma(0)=x$. Then the curve $\Phi\circ \gamma$ is the straight line passing through the origin with velocity~$\dot{\gamma}(0)$, which immediately gives~$\Phi(p)\horismos \Phi(q)$, a contradiction to~$\Phi(p)\chron \Phi(q)$.
\end{proof}
}

\begin{lemmarep}\label{lemma:unique lightlines in any spacetime}
    In any spacetime, if $x\after y \after y_1,y_2$ and $y_1\neq y_2$ are spacelike separated, then $x\chron y_i$ for some $i$.
\end{lemmarep}
\begin{proof}
\change{
    Let $\gamma,\delta_1,\delta_2$ denote the causal curves corresponding to $x\after y$ and $y\after y_1,y_2$, respectively. Without loss of generality, suppose that $\delta_i$ are parametrised on the unit interval $[0,1]\to M$.
    Consider the set $S= \delta_1^{-1}(\im(\delta_2))$, which is nonempty since $\delta_1(0)=\delta_2(0)=y$ and a strict subset of $[0,1]$ because~$y_1\neq y_2$ are spacelike separated.
    Since~$[0,1]$ is compact and $\delta_2$ is continuous,~$\im(\delta_2)$ is compact in~$M$. Since~$M$ is Hausdorff,~$\im(\delta_2)$ is thus closed, and by continuity of $\delta_1$ the preimage~$S$ is closed. Hence $S$ has a maximum element $t_0$, the latest time at which~$\delta_1$ overlaps~$\delta_2$. We can further reparametrise $\delta_2$ without loss of generality such that $\delta_1(t_0)=\delta_2(t_0)$, and hence $\delta_1$ and $\delta_2$ never intersect after~$t_0$. Let~$y':= \delta_1(t_0)=\delta_2(t_0)$ be the last intersection point; see the left diagram in Figure~\ref{fig:uniquelightlines}.
    Note if~$t_0=0$ then~$y'= y$.
    
    Using Theorem~\ref{thm:GNC}, choose a convex normal neighbourhood $U_{y'}\subseteq M$ and geodesic normal coordinates
    $\Phi\colon U_{y'}\to \mathcal{U}_{y'}$.
    Since $U_{y'}$ is open and $\delta_i$ are continuous with $\delta_i(t_0)=y'$, we can find time intervals $\epsilon_i>0$ such that $t_0+\epsilon_i\leq 1$ and on $[t_0,t_0+\epsilon_i]$ the curves $\delta_i$ stay in $U_{y'}$.
    For any~${t\in (t_0,t_0+\epsilon_i]}$ we then get that $y'\after_{U_{y'}} \delta_i(t)$, and by~\eqref{eq:local relations} we have $0\after \Phi\circ \delta_i(t)$ in $\mathcal{U}_{y'}$.
    From this it follows that~$\delta_i(t) \neq y'$ for every~$t\in (t_0,t_0+\epsilon_i]$, otherwise $y'\after_{U_{y'}} y'$ and hence $0\after 0$, which is impossible in Minkowski space.
    Using this, we see from Equation~\eqref{eq:caus} that the time coordinate~${(\Phi\circ \delta_i(t))^0}$ must be strictly greater than zero for all~$t\in (t_0,t_0+\epsilon_i]$.
    In particular there exists some~$\epsilon >0$ that is below the endpoints:~${\epsilon < \min_i (\Phi\circ \delta_i(t_0+\epsilon_i))^0}$, meaning that eventually the curves $\Phi\circ \delta_i$ are above the hyperplane~${\{ x\in \mathbb{R}^{1+n}: x^0= \epsilon\}}$.
    By the intermediate value theorem there exists~${t_i\in (t_0,t_0+\epsilon_i]}$ so that the curves are precisely on the hyperplane: $(\Phi\circ\delta_i(t_i))^0 = \epsilon$.
    Using this, define the distinct points~$p_i:=\delta_i(t_i)$, from which we get the distinct points~$z_i:=\Phi(p_i)$ on the hypersurface, meaning that $z_1\neq z_2$ are spacelike separated.
    
    We similarly pick a point $p$ in the past of $y'$. If $t_0>0$, find some~$\epsilon_0>0$ so that $\delta_1$ remains within~$U_{y'}$ on the interval $[t_0-\epsilon_0,t_0]$, and define $p:= \delta_1(t_0-\epsilon_0)$ and~$z:=\Phi(p)$; see the right diagram in Figure~\ref{fig:uniquelightlines}. In the case that $t_0=0$, so $y'=y$, pick $p$ analogously but on the curve $\gamma$ instead.

    In total, we have the situation in the normal neighbourhood $U_{y'}$ that $p\after_{U_{y'}} y' \after_{U_{y'}} p_1,p_2$. Using~\eqref{eq:local relations} and its dual, we get the situation in the region $\mathcal{U}_{y'}$ of Minkowski space that~$z\after 0 \after z_1,z_2$ where~$z_1\neq z_2$ are spacelike separated, so by~Corollary~\ref{corollary:unique lightline in Minkowski} we get~$z\chron z_k$ for some~$k$.
    Applying Lemma~\ref{lem:lift chronology} we hence get $p\chron_{U_{y'}} p_k$.

    Finally, by construction we have $x\after p$ and $p_i\after y_i$, so globally $x\after p \chron p_k \after y_k$, and the desired result follows by the push-up rule:~$x\chron y_k$.
}
\end{proof}

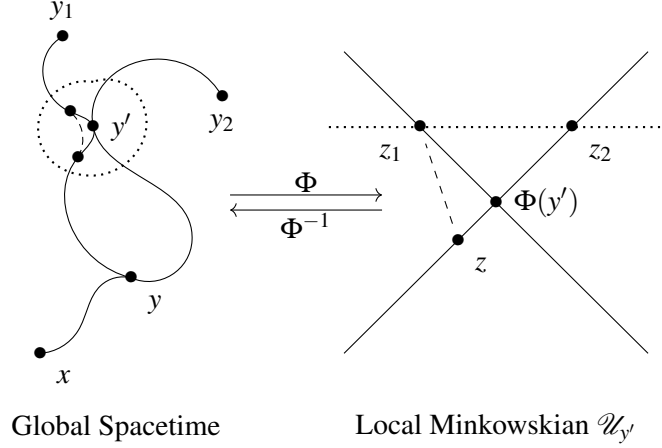
\begin{figure}[t]
    \centering
    \begin{tikzpicture}
    \node[] (gen) at (1,-1) {Global Spacetime};
    \node[label=below right:{$x$}] (x) {$\bullet$};
    \node[label=below right:{$y$}] (y) at (1.2,1) {$\bullet$};
    \draw [] (x.center) to [ curve through ={(.5,.3) . . (.8,.9)  }] (y.center);
    
    \node[label=right:{$y'$}] (yp) at (.7,3) {$\bullet$};
    \node[label=above:{$y_1$}] (y1) at (.3,4.2) {$\bullet$};
    \node[label=below:{$y_2$}] (y2) at (2.4,3.4) {$\bullet$};
    \node[] (diz) at (.5,2.6) {$\bullet$}; % label=left:{$\Phi^{-1}(z)$}
    \node[] (diz1) at (.4,3.2) {$\bullet$}; % label=left:{$\Phi^{-1}(z_1)$}
    \node[] (diz2) at (.7, 3.3) {$\bullet$}; % label=left:{$\Phi^{-1}(z_2)$}
    \draw [] (y.center) to [ curve through ={(1,1.5) . . (.4,1.7)  . . (diz) . . (yp) . . (diz1) . . (.1,4)  }] (y1.center);
    \draw [] (y.center) to [ curve through ={(2,1.5)  . . (yp) . .  (diz2) . . (1.6,3.9) }] (y2.center);
    \draw [densely dashed] (diz.center) to [ curve through ={(.55,2.95)}] (diz1.center);
    \node[] (U) at (.7,3.6) {};
    \draw [dotted,thick] (U.center) to [ curve through ={(0,2.8)  . . (.4,2.4) . . (1.3,3.4)  }] (U.center);

    \node[label=above:{$\Phi$},label=below:{$\Phi^{-1}$}] at (3.5,2) {};
    \draw[->] (2.5,2.1) -- (4.5,2.1);
    \draw[->] (4.5,1.9) -- (2.5,1.9);

    \node[] (mink) at (6, -1) {Local Minkowskian $\mathcal{U}_{y'}$};
    \node[label=right:{$\Phi(y')$}] (dyp) at (6, 2) {$\bullet$};
    \draw[] (4,0) -- (dyp.center) -- (4, 4);
    \draw[] (8,0) -- (dyp.center) -- (8, 4);
    \node[label=below right:{$z$}] (z) at (5.5,1.5) {$\bullet$};
    \node[label=below left:{$z_1$}] (z1) at (5,3) {$\bullet$};
    \node[label=below right:{$z_2$}] (z2) at (7,3) {$\bullet$};
    \draw[dashed] (z) -- (z1);
    \draw[thick, dotted] (3.8,3) -- (8.2,3);
\end{tikzpicture}
    \caption{Visualisation of the proof for Lemma~\ref{lemma:unique lightlines in any spacetime} using two-dimensional Minkowski spacetime.
    Dashed (solid) lines represent chronological~(causal)~curves. The dotted region on the left denotes~$U_{y'}$, and the dotted line on the right depicts the spacelike hypersurface defined by $x^0=\epsilon$.}
    \label{fig:uniquelightlines}
\end{figure}

\change{We are now ready to prove the main theorem.}

\begin{theoremrep}\label{thm:spacetimeafter}
    Let $\genericworld$ be a spacetime and $\after$ its after relation. Then
    \begin{equation*}
        \entail{\mframe{\genericworld}{\after}}{}{a\alpha f}.
    \end{equation*}
\end{theoremrep}
\begin{proof}
    Take $x, y, y_1, y_2, z \in M$ such that $x \after y, y_1, y_2, z$ and $y \after y_1, y_2$, where $y_1 \neq y_2$ are unrelated by~$\after$. This implies $y_1$ and $y_2$ are spacelike separated. We need to find $t\in\genericworld$ such that $x\after t \after y_i, z$, for some~$i$. First, by Lemma~\ref{lemma:unique lightlines in any spacetime} we can assume without loss of generality that $x\chron y_1$. In the case~$x=z$ the result is trivial~(pick $t=x$), so assume~$x\neq z$. Thus we have a case distinction of whether $x\chron z$ or $x\horismos z$ (Lemma~\ref{lemma:relations between causal orders}). In the first case, the existence of $t\in \genericworld \setminus \{x\}$ with $x\chron t\chron y_1,z$ follows immediately from 2-density of $\chron$.
    
    Instead, suppose that $x \horismos z$. Let~$\gamma$ be the timelike curve connecting $x\chron y_1$ and let~$\delta$ be the lightcurve connecting $x\horismos z$. Take a point $p\in \im(\gamma)\setminus\{x\}$ so that $x\chron p\chron y_1$.
    Then $x\in I^{-}(p)$, and since chronological cones are open there must exist a point $t\in I^{-}(p)\cap \im(\delta)\setminus\{x\}$, which hence satisfies $x\after t \after y_1$ and $x \after t \after z$, as desired.
\end{proof}

Finally, we build on Corollary~\ref{cor:spacetimeclass} using Theorem~\ref{thm:spacetimeafter} and Lemma~\ref{lem:afterformula}.
\begin{corollary}
    For any spacetime $\genericworld$ we have that $\afterlogic \subseteq \modal{\mframe{\genericworld}{\after}}$.
\end{corollary}

What this section shows is that there is a connection between the push-up rule, which naturally occurs in spacetimes, and $\afterformula$.
Many of the proofs that we have used for the $\afterformula$ are reliant on the push-up rule.
An exploration of the push-up rule and its link to Kripke frames (possibly with multiple relations) is warranted, but left to future work.

\subsection{The After Formula in Finite Frames}
\label{sec:moreafter}
To begin with, we introduce some concepts for modal logics.
A cluster, $C_x$, of a point, $x$, is a set containing $x$ and the set of points that can see and be seen by $x$, \ie{} $C_x = \{ x \} \cup \{y \in \genericworld : y \genericrelation x, x \genericrelation y \}$.
A degenerate cluster is a single irreflexive point.
A successor cluster of a point, $x \in \genericworld$, is a cluster,~$C$, where $x \genericrelation y$ for $y \in C$ and if $x \genericrelation y' \genericrelation y$, then $y' \in C$ or $y' \in C_x$ (\ie{} there is no other cluster between~$x$ and~$C$).\footnote{Note that in transitive frames and when $x$ is irreflexive, then $C_x = \{ x \}$ and we must have that $y' \in C$.}
For $x \in \genericworld$, denote $S_x = \{C : C~\text{is a successor cluster of}~x \}$.

We analyse the effect $\afterformula$ has on irreflexive points in a finite frame.
For simplicity sake, we will assume that the frame is transitive and dense (as are all the causal relations).
Reflexive points are not interesting since $\afterformula$ will always hold in those points.
\begin{propositionrep}
Let $F = \mframe{\genericworld}{\genericrelation}$ be a finite frame that is transitive and dense.
Let $x \in \genericworld$ be irreflexive and $S_x$ the set of (non-degenerate) successor clusters of $x$.
Then $\mentail{F}{}{\afterformula}$
iff
whenever $S \in S_x$ and $y_1, y_2 \in \genericrelation(S)$ with $y_1 \neq y_2$, $y_1 \ngenericrelation y_2$, $y_2 \ngenericrelation y_1$, then $\bigcup_{S \in S_x} S \subseteq \genericrelation^{-1}(y_1) \cup \genericrelation^{-1}(y_2)$.
\label{prop:successorcluster}
\end{propositionrep}
\begin{proof}
    $\Leftarrow$:
    Let $V$ be an evaluation on $F$ and consider the model $\genericmodel$.
    Firstly note that for any $x \in \genericworld$ that is reflexive, we have that $\mentail{\genericmodel}{x}{\afterformula}$.
    
    Now suppose $x \in \genericworld$ is an irreflexive point and that 
    \begin{equation*}
    \mentail{\genericmodel}{x}{\poss ( \poss (p_1 \land \lnot p_2 \land \necc \lnot p_2) \land \poss (p_2 \land \lnot p_1 \land \necc \lnot p_1)) \land \poss q}.    
    \end{equation*}
    There exists a point $y \in \genericrelation(x)$ such that $\mentail{\genericmodel}{y}{\poss (p_1 \land \lnot p_2 \land \necc \lnot p_2) \land \poss (p_2 \land \lnot p_1 \land \necc \lnot p_1)}$.
    We note that for some cluster $S \in S_x$, $y \in S$ or $y \in \genericrelation(S)$.
    This is because if $y$ were not in a successor cluster after $x$ and not in the future of a successor cluster, then its own cluster would be a successor cluster, which is a contradiction.
    
    Additionally, we have $y_1, y_2 \in \genericrelation(y)~(\subseteq \genericrelation(S))$ such that $\mentail{\genericmodel}{y_1}{p_1 \land \lnot p_2 \land \necc \lnot p_2}$ and $\mentail{\genericmodel}{y_2}{p_2 \land \lnot p_1 \land \necc \lnot p_1}$.
    It can be seen that $y_1 \neq y_2$, as $p_1$ holds in $y_1$ but not in $y_2$; and that $y_1 \ngenericrelation y_2$, as $\necc \lnot p_2$ holds in $y_1$ but $p_2$ holds in $y_2$.
    Similarly, $y_2 \ngenericrelation y_1$.
    Thus, we have that $\bigcup_{S \in S_x} S \subseteq \genericrelation^{-1}(y_1) \cup \genericrelation^{-1}(y_2)$.

    Since $\mentail{\genericmodel}{x}{\poss q}$, there exists $z \in \genericrelation(x)$ such that $\mentail{\genericmodel}{z}{q}$.
    Again, there exists some $S' \in S_x$ such that either $z \in S'$ or $z \in \genericrelation(S')$.

    Let $t \in S'$.
    We have that $t \in \genericrelation^{-1}(y_1) \cup \genericrelation^{-1}(y_2)$ and therefore $y_1 \in \genericrelation(t)$ or $y_2 \in \genericrelation(t)$.
    Thus, we must have that $\mentail{\genericmodel}{t}{\poss p_1 \lor \poss p_2}$.
    As $z \in S'$ or $z \in \genericrelation(S')$, then $z \in \genericrelation(t)$ and $\mentail{\genericmodel}{t}{\poss q}$, which means that $\mentail{\genericmodel}{t}{(\poss p_1 \land \poss q) \lor (\poss p_2 \land \poss q)}$.
    Finally, since $t \in \genericrelation(x)$, we have
    \begin{equation*}
        \mentail{\genericmodel}{x}{\poss(\poss p_1 \land \poss q) \lor \poss (\poss p_2 \land \poss q))}.
    \end{equation*}
    Therefore, $\mentail{\genericmodel}{x}{\afterformula}$.
    As $\afterformula$ holds for any $x \in \genericworld$ and $V$ is generic, then $\mentail{F}{}{\afterformula}$.
    
{
    \newcommand{\genericmodelt}{\mmodel{\genericworld}{\genericrelation}{\genericeval_T}}
    $\Rightarrow$:
    Suppose now that $\mentail{F}{}{\afterformula}$.
    Consider $y_{1}, y_{2} \in \genericrelation(S)$ with $y_{1} \neq y_{2}$, $y_{1} \ngenericrelation y_{2}$, $y_{2} \ngenericrelation y_{1}$.

    We will be creating a generic model of $F$.
    Let $\genericeval$ be an evaluation such that $\genericeval(p_1) = \{y_{1}\}$, $\genericeval(p_2) = \{y_{2}\}$.
    For $T \in S_x$, let $\genericeval_T$ be an evaluation where $\genericeval_T(q) = T$ and $\genericeval_T(p_i) = \genericeval(p_i)$.
    It should be clear that
    \begin{equation*}
        \mentail{\genericmodelt}{x}{\poss ( \poss (p_1 \land \lnot p_2 \land \necc \lnot p_2) \land \poss (p_2 \land \lnot p_1 \land \necc \lnot p_1)) \land \poss q},
    \end{equation*}
    and since $\mentail{\genericmodelt}{x}{\afterformula}$, then $\mentail{\genericmodelt}{x}{\poss(\poss p_1 \land \poss q) \lor \poss(\poss p_2 \land \poss q)}$.

    Since $T$ is a successor cluster of $x$, then $\poss q$ holds in $T$ or $x$ (and no other points in $\genericrelation(x)$).
    As $x$ is irreflexive, $\exists t \in T$ such that $\mentail{\genericmodelt}{t}{\poss p_1 \land \poss q}$ and, as $T$ is a non-degenerate cluster, for any $t' \in T$, $\mentail{\genericmodelt}{t'}{\poss p_1}$ by transitivity.
    The same reasoning holds for $p_2$, so we must have that $t' \genericrelation y_{1}$ or $t' \genericrelation y_{2}$.
    We have $T \subseteq \genericrelation^{-1}(y_{1}) \cup \genericrelation^{-1}(y_{2})$ and since $T$ was generically chosen, $\bigcup_{S \in S_x} S \subseteq \genericrelation^{-1}(y_{1}) \cup \genericrelation^{-1}(y_{2})$.
    }
\end{proof}

In the class of transitive, dense, and $\afterformula$ frames, any two points after a successor cluster of an irreflexive point must be jointly connected between all other successor clusters.
We can push this concept further to find an interesting sub-class.
We begin with a definition.
\begin{definition}
    A chain of clusters in $\genericframe$ is a tuple of clusters, $(C_1, C_2, \dots, C_k)$, such that for any $c \in C_l$, $\genericrelation^{-1}(c) = \bigcup_{i=1}^{l}(C_i)$ if $C_l$ is non-degenerate and $\genericrelation^{-1}(c) = \bigcup_{i=1}^{l-1}(C_i)$ otherwise (\emph{i.e.,} within $\genericframe$ a point in a cluster can only have seen the previous clusters).
\end{definition}

\begin{propositionrep}
    Let $F = \mframe{\genericworld}{\genericrelation}$ be a finite frame that is transitive, dense, and has the property of $\afterformula$.
    Let $x \in \genericworld$ an irreflexive point with $R(x) \neq \varnothing$ and $|S_x| \geq 2$.
    Let $F_x = (\genericworld_x, \genericrelation_{x})$ be the generated subframe of $x$, where $\genericworld_x = \{x\} \cup \genericrelation(x)$ and $\genericrelation_{x} = \genericrelation \cap (\genericworld_x \times \genericworld_x)$.
    For any (non-degenerate) successor cluster of $x$, $S \in S_x$, either it has no chain of clusters in $\mframe{\genericworld_x}{\genericrelation_{x}}$ or it has a unique chain of clusters in $\mframe{\genericworld_x}{\genericrelation_x}$, $(\{x\}, S, S_1, \dots S_k)$ where $S_i \subseteq \genericworld_x$, such that if $y \in \genericrelation_{x}(S)$ and $y \notin S \cup S_1 \cup \dots \cup S_k$, then $\bigcup_{S' \in S_x} S' \subseteq \genericrelation^{-1}(y)$.
    \label{cor:uniquechain}
\end{propositionrep}
\begin{proof}
    Suppose that $S \in S_x$ has more than one chain of clusters associated, denoted $(\{x\},S, S_1, \dots, S_k)$ and $(\{x\},S, S'_1, \dots, S'_k)$.
    Let $c \in S_{i}$ and $c' \in S'_{j}$.
    We have that $c \neq c'$, $c \ngenericrelation c'$, $c' \ngenericrelation c$, and $c, c' \in \genericrelation(S)$.
    By Proposition~\ref{prop:successorcluster}, we have that $\bigcup_{S' \in S_x} S' \subseteq \genericrelation(c) \cup \genericrelation(c')$.

    Select $T \in S_x$ such that $T \neq S$.
    It must be the case that $T \in \genericrelation^{-1}(c)$ or $T \in \genericrelation^{-1}(c')$, and therefore $T \in \genericrelation_x^{-1}(c)$ or $\genericrelation_{x}^{-1}(c')$.
    This is a contradiction as $c$ and $c'$ are in a chain of clusters, and so $T$ cannot be within $\genericrelation_{x}^{-1}(c)$ or $\genericrelation_{x}^{-1}(c)$.
    Therefore, $S$ must only be in one or no chain of clusters.

    Now suppose that $S$ is associated with a chain of clusters and let $y \in \genericrelation_{x}(S)$ with $y \notin S \cup S_1 \cup \dots \cup S_k$.
    We know that for any $c \in S \cup \bigcup_i S_i$ we must have that $\bigcup_{S' \in S_x} S' \subseteq \genericrelation^{-1}(c) \cup \genericrelation^{-1}(y)$.
    However, $c$ can only see points in $S$ since it is in a chain of clusters and cannot see any points in any other successor cluster.
    Therefore, it must be the case that $y$ must see all points in all the other clusters.
    As we have $S \subseteq \genericrelation^{-1}(y)$, then $\bigcup_{S' \in S_x} S' \subseteq \genericrelation^{-1}(y)$
\end{proof}

Note that Proposition~\ref{cor:uniquechain} only guarantees a chain of clusters within the generated subframe.
The past of a point within the chain of clusters can be larger within the full frame and may branch outside of the subframe.

These chains of clusters are similar to the light lines in a spacetime.
One can remain only on a single chain of clusters (light line) and once you leave the chain you are in the ``interior''.
This interior can be reached by any of the light lines and acts like the chronological relation.
Additionally, once you enter the interior, you can never return to the outer chains.
This is demonstrated in a frame given in Figure~\ref{fig:exafterformulaframe}.

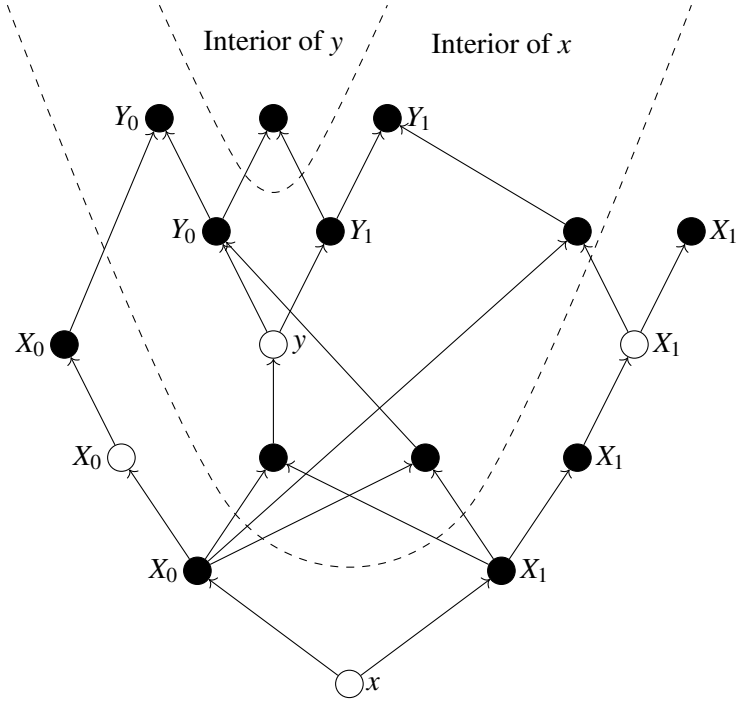
\begin{figure}[t!]
    \centering
    \begin{tikzpicture}[
        nodes={draw, circle}, ->,
        level 1/.style={sibling distance=4cm},
        level 2/.style={sibling distance=2cm},
        level 3/.style={sibling distance=1.5cm},
        level 4/.style={sibling distance=1.5cm},
    ]
    \node[label=right:{$x$}] (s) {} [grow'=up]
        child {node[refl,label=left:{$X_0$}] (s0) {}
            child {node[label=left:{$X_0$}] (s00) {}
                child {node[refl,label=left:{$X_0$}] (s000) {}}
                child {edge from parent[draw=none]}
            }
            child {node[refl] (i0) {}
                child {node[] (i1) [label=right:{$y$}] {}
                    child {node[refl,label=left:{$Y_0$}] (i1s0) {}
                        child {node[refl,label=left:{$Y_0$}] (i1s00) {}}
                        child {node[refl] (i1i0) {}
                            % child {node[empty] (i1i1) {$\vdots$}}
                            % child {node[empty] (i1i2) {$\vdots$}}
                            }
                    }
                    child {node[refl,label=right:{$Y_1$}] (i1s1) {}
                        child {edge from parent[draw=none]}
                        child {node[refl, label=right:{$Y_1$}] (i1s10) {}}
                        }
                }
            }
        }
        child {node[refl,label=right:{$X_1$}] (s1) {}
            child {node[refl] (i2) {}}
            child {node[refl,label=right:{$X_1$}] (s10) {}
                child {edge from parent[draw=none]}
                child {node[label=right:{$X_1$}] (s100) {}
                    child {node[refl] (i3) {}
                        % child {node[empty] (i30) {$\vdots$}}
                    }
                    child {node[refl,label=right:{$X_1$}] (s1000) {}} 
                }
            }
        };
        \node[empty] () at (2,8.5) {Interior of $x$};
        \node[empty] () at (-1,8.5) {Interior of $y$};
        % \node[empty] () at (-4,2) {First chain of $x$};
        % \node[empty] () at (5,3.5) {Second chain of $x$};
        
        \draw (s1) -- (i0);
        \draw (s0) -- (i2);
        \draw (i2) -- (i1s0);
        \draw (s0) -- (i3);
        \draw (i3) -- (i1s10);
        \draw (s000) -- (i1s00);
        \node[empty] (m0) at ($(s0)!0.5!(i0)$) {};
        \node[empty] (m1) at ($(s1)!0.5!(i2)$) {};

        \draw (i1s1) -- (i1i0);
        % \draw[dashed,-] (-2,7) to [bend right = 20] (m0.center) to [bend right=20] (m1.center) to [bend right = 20] (2,7);
        \draw[dashed,-] plot [smooth] coordinates {(-4.5,9) (m0.center) (m1.center) (4.5,9)};
        \draw[dashed,-] plot [smooth] coordinates {(-2.5,9) ($(i1s0)!0.5!(i1i0)$) ($(i1s1)!0.5!(i1i0)$) (0.5,9)};
    \end{tikzpicture}
    \caption{A transitive, dense frame with $\afterformula$.
    The points $x$ and $y$ each have two chains of clusters within their own generated subframe (labelled $X_0, X_1, Y_0, Y_1$).
    Note how the points on the chains of clusters starting from $y$ can have points see them as long as $y$ does not see those points.}
    \label{fig:exafterformulaframe}
\end{figure}

\shelf{Finite model property/Decidability of $\afterlogic$?}

% \subsection{Undefinability of Domain}
% \input{Sections/CommonProperties/domain}

\section{Two-dimensional Spacetimes}
\label{sec:separable}
It remains an open problem as to what the modal logic of Minkowski spacetime with $\after$ is, but it is known that the logic of two-dimensional Minkowski spacetime is unique to any $1+n$-dimensional Minkowski spacetime ($n > 2$)~\cite{Shapirovsky2005}.

\subsection{After Formula for Two Dimensions}
\label{sec:separable:after2}
Consider the formula,
\begin{equation*}
    \aftertwoformula := \poss (p_1 \land \lnot p_2 \land \necc \lnot p_2)
                \land \poss (p_2 \land \lnot p_1 \land \necc \lnot p_1)
                \land \poss q
                \Rightarrow
                \poss (\poss p_1 \land \poss q) \lor \poss (\poss p_2 \land \poss q),
\end{equation*}
which is $\afterformula$ without the first diamond.\footnote{Read as the after-2 formula/axiom.}
It corresponds to the following first order property:
\begin{equation*}
\begin{aligned}
    \forall x, y_1, y_2, z & \Big[ [x \genericrelation y_1 \land x \genericrelation y_2 \land x \genericrelation z % \\ & \qquad
    \land y_1 \neq y_2 \land \lnot(y_1  \genericrelation  y_2) \land \lnot(y_2  \genericrelation  y_1) ]
    \\ &  \quad \to
    \exists t (x \genericrelation t \land t \genericrelation z \land (t \genericrelation y_1 \lor t \genericrelation y_2 )) \Big].
\end{aligned}
\end{equation*}
This formula is a Sahlqvist formula and, therefore, any canonical normal modal logic that adds $\aftertwoformula$ is canonical.
Even though $\aftertwoformula$ is a simple modification to $\afterformula$, we can show that it is strictly more expressive within the context of spacetimes.

\begin{theoremrep}
We have that
    \begin{enumerate}[label=(\roman*)]
        \item $\mathbf{K4} + \aftertwoformula \vdash \afterformula$; \label{thm:aftertwoexpress:expressive}
        \item $\afterlogic \nvdash \aftertwoformula$. \label{thm:aftertwoexpress:inexpressive}
    \end{enumerate}
    \label{thm:aftertwoexpress}
\end{theoremrep}
\begin{proof}
    \ref{thm:aftertwoexpress:expressive}:
    Let $\mathcal{M} = \mmodel{\genericworld}{\genericrelation}{\genericeval}$ be a $\mathbf{K4} + \aftertwoformula$ model and $x \in \genericworld$.
    Assume that $\mentail{\mathcal{M}}{x}{\poss(\poss (p_1 \land \lnot p_2 \land \lnot \necc p_2) \land \poss(p_2 \land \lnot p_1 \land \lnot \necc p_1)) \land \poss q}$.
    Let $x \genericrelation y$ such that $\mentail{\mathcal{M}}{y}{\poss (p_1 \land \lnot p_2 \land \lnot \necc p_2) \land \poss(p_2 \land \lnot p_1 \land \lnot \necc p_1)}$, then $\mentail{\mathcal{M}}{x}{\poss \poss (p_1 \land \lnot p_2 \land \lnot \necc p_2) \land \poss\poss(p_2 \land \lnot p_1 \land \lnot \necc p_1)}$.
    
    Since $\mathcal{M}$ is transitive, then we have
    $\mentail{\mathcal{M}}{x}{\poss (p_1 \land \lnot p_2 \land \lnot \necc p_2) \land \poss(p_2 \land \lnot p_1 \land \lnot \necc p_1) \land \poss q}$.
    As $\mathcal{M}$ is a $\aftertwoformula$ model, we have $\mentail{\mathcal{M}}{x}{\poss (\poss p_1 \land \poss q) \lor \poss (\poss p_2 \land \poss q)}$ and therefore $\mentail{\mathcal{M}}{x}{\afterformula}$.
    Because $x$ and $\mathcal{M}$ are general, then $\mathbf{K4} + \aftertwoformula \vdash \afterformula$.

    \begin{figure}[t]
        \centering
        \begin{tikzpicture}[nodes={draw, circle}, ->,]
        \node (s) {} [grow'=up]
        child {node[refl] (0) [label=left:{$p_1$}] {}}
        child {node[refl] (1) [label=left:{$p_2$}] {} }
        child {node[refl] (2) [label=left:{$q$}] {} };
        \end{tikzpicture}
        \caption{A finite $\afterlogic$ model that does not have $\aftertwoformula$ on the root.}
        \label{fig:aftertwoexpress}
    \end{figure}
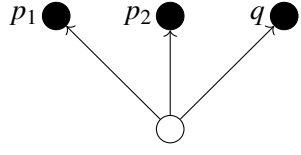
    \ref{thm:aftertwoexpress:inexpressive}: A counterexample is provided by the finite model given in Figure~\ref{fig:aftertwoexpress}.
    The frame has all the properties of $\afterlogic$, but cannot express $\aftertwoformula$.
    One of the points $p_1$ or $p_2$ must share a common point with $q$ after the root but they do not.
\end{proof}

\begin{remark}
For finite (spacetime) frames, $\aftertwoformula$ says that if there are successor clusters from an irreflexive point, there are at most two of them each possibly with their own unique chains of clusters (light lines).
This can be shown in a similar way as in Section~\ref{sec:moreafter}.
\end{remark}

\subsection{Separability of Two-Dimensional Spacetimes}
\label{sec:separable:separability}
With the after relation, we can see that $\aftertwoformula$ holds in two-dimensional Minkowski spacetime but does not hold in any other $1+n$-dimensional Minkowski spacetime.

\begin{lemma}
    \label{lemma:aftertwoinminkowski}
    Let $M = \mathbb{R}^{1+n}$ and $\after$ be the after relation on Minkowski spacetime.
    We have that
    \begin{enumerate}[label=(\roman*)]
        \item $\mentail{\mframe{\mathbb{R}^2}{\after}}{}{\aftertwoformula}$
        \item $\nmentail{\mframe{\mathbb{R}^{1+n}}{\after}}{}{\aftertwoformula}$ for $n \geq 2$
    \end{enumerate}
\end{lemma}
\begin{proof}
The argument is similar to~\cite[Lemma 6.2]{Shapirovsky2005}, but we give an outline of the proof.
In $\mathbb{R}^2$, there are only two light lines where points can be and so a third point will either be on a light line with a point already on it or in the interior (so there will always be a common point, $t$); but in $\mathbb{R}^n$, one can pick points on three separate light lines and there will be no common point between the three of them after the root point, $x$.
\end{proof}

It turns out that $\aftertwoformula$ not only holds in two-dimensional Minkowski spacetime, but any two-dimensional spacetime.

\begin{theoremrep}
Let $M$ be a two-dimensional spacetime.
Then, 
\begin{equation*}
\mentail{\mframe{M}{\after}}{}{\aftertwoformula}.    
\end{equation*}
\end{theoremrep}
\begin{proof}
\change{
Let $x \after y_1, y_2, z$ with $y_1 \neq y_2$ and $y_i \nafter y_j$.
We need to prove there exists $t\in M$ such that $x\after t\after z, y_i$ for some $i$.
Again, in the case that $x=y_i$ or $x=z$ simply take $t=x$. Thus assume all points are distinct from~$x$, so $\after$ is described via Lemma~\ref{lemma:relations between causal orders}.

If $x \chron y_i$ for some $i$ and $x \chron z$, then by 2-density of $\chron$ we are done.

Suppose now that $x \horismos y_1, y_2, z$, and let $\delta_1,\delta_2,\gamma$ be the corresponding causal curves.
Since $x\horismos y_i$, there are no timelike curves from $x$ to $y_i$, and hence by Theorem~\ref{thm:2.22} each $\delta_i$ is an achronal lightlike geodesic (up to reparametrisation).
Similarly, since $x\horismos z$, the curve $\gamma$ is an achronal lightlike geodesic from $x$ to $z$.

Choose a convex normal neighbourhood $U_x$ of $x$, and consider the restrictions of $\delta_1,\delta_2,\gamma$ to~$U_x$. In a two-dimensional spacetime there are exactly two future-directed lightlike velocities at $x$, and hence exactly two future-directed lightlike geodesics emanating from $x$ in $U_x$.
Moreover, since $y_1$ and $y_2$ are spacelike separated, the geodesics $\delta_1$ and $\delta_2$ cannot coincide, and so the velocities of $\delta_1$ and $\delta_2$ at~$x$ correspond to the two distinct lightlike directions in $T_xM$.
Since $\gamma$ is also lightlike, its velocity at~$x$ must therefore be equal to that of, say, $\delta_1$.
Pick a point $t\in \im(\gamma)\cap \im(\delta_1)\cap U_x\setminus\{x\}$ sufficiently close to $x$ so that $t$ lies on the subcurves from $x$ to $z$ and from $x$ to $y_1$, so that $x\after t\after y_1,z$, as desired.

Finally, consider the scenario $x \horismos y_1, y_2$ and $x \chron z$.
Let $\delta_1,\delta_2$ be the achronal lightlike geodesics from $x$ to $y_1,y_2$ (as above), and let $\gamma$ be a timelike curve from $x$ to $z$.
Pick a point ${z'\in \im(\gamma)\setminus\{x,z\}}$ sufficiently close to $x$ so that $x\chron z'\chron z$.
Since $I^{-}(z')$ is open and contains $x$, and $\delta_1$ is continuous with $\delta_1(0)=x$, there exists ${t\in \im(\delta_1)\setminus\{x\}}$ sufficiently close to $x$ such that $t\in I^{-}(z')$.
Then ${x\after t\after y_1}$. Moreover $t\chron z'$ implies $t\after z'\after z$, hence $t\after z$.
Thus $x\after t\after y_1,z$, completing the proof.
}
\end{proof}

Denote $\aftertwologic := \mathbf{D4} + ad + \aftertwoformula$.
It is clear from Theorem~\ref{thm:aftertwoexpress} that $\afterlogic \subset \aftertwologic$.
% showing modal logic is expressible enough to differentiate between 2-dimensional and $1+n$-dimensional spacetimes.
\begin{corollary}
    Let $M$ be a $2$-dimensional spacetime. Then $\aftertwologic \subseteq \modal{\mframe{M}{\after}}$.
\end{corollary}

In Lemma~\ref{lemma:aftertwoinminkowski}, we were shown how $\aftertwoformula$ only holds in two dimensions, providing a separation from $1+n$-dimensional Minkowski spacetime.
However, to show a full separation of the logics of 2-dimensional spacetimes from higher dimensional ones, it needs to be shown that $\aftertwoformula$ does not hold in higher dimensional spacetimes.
We note that non-cNTV spacetimes are reflexive in $\after$, meaning that $\aftertwoformula$ holds in those spacetimes.
However, for cNTV spacetimes of dimension greater than 3, it is likely that $\aftertwoformula$ does not hold.
Proving this requires that any high dimensional, cNTV spacetime has at least 3 different light lines from a point irreflexive in $\after$.

\shelf{Prove $\aftertwoformula$ does not hold in $n \geq 3$ spacetime, i.e., $\modal{\mframe{M}{\after}} \subset \modal{\mframe{M}{\after}} + \aftertwoformula$}

\begin{conjecture}
    Let $n \geq 3$.
    For any cNTV $n$-dimensional spacetime, $M$, we have that
    $\nmentail{\mframe{M}{\after}}{}{\aftertwoformula}$ ($\aftertwologic \nsubseteq \modal{\mframe{M}{\after}}$).
\end{conjecture}

\subsection{3,2-Density, After, and two-dimensional Minkowski Spacetime}
\label{sec:separable:minkowski2}
The reasoning for $\aftertwoformula$ holding (or not) in Minkowski spacetime is the same for \emph{3,2-density}, described by the formula
\begin{equation*}
    ad_{3,2} := \poss p_1 \land \poss p_2 \land \poss p_3
                \Rightarrow
                \poss (\poss p_1 \land \poss p_2) \lor \poss (\poss p_1 \land \poss p_3) \lor \poss (\poss p_2 \land \poss p_3),
\end{equation*}
which was shown in~\cite{Shapirovsky2005} to hold in two-dimensional Minkowski spacetime but not in higher $1+n$-dimensional Minkowski spacetime ($n \geq 2$).
It was also shown therein that $\mK + ad_{3,2} \vdash ad$.

Here we show that there is a close relation between $ad_{3,2}$ and $\aftertwoformula$.
\begin{theorem}
    We have that
    \begin{enumerate}[label=(\roman*)]
        \item $\mathbf{K4} + ad + \aftertwoformula \vdash ad_{3,2}$; \label{thm:aftertwod32:d32}
        \item $\mathbf{K4} + ad_{3,2} \nvdash \aftertwoformula$; and \label{thm:aftertwod32:notaftertwo}
        \item $\mathbf{K4} + \afterformula + ad_{3,2} \vdash \aftertwoformula$. \label{thm:aftertwod32:afterd32}
    \end{enumerate}
\end{theorem}
\begin{proof}
\ref{thm:aftertwod32:d32}: let $\mathcal{M} = \mmodel{\genericworld}{\genericrelation}{\genericeval}$ be a $\mathbf{K4} + ad + \aftertwoformula$ model and $x \in \genericworld$ such that $\mentail{\mathcal{M}}{x}{\poss p_1 \land \poss p_2 \land \poss p_3}$.
We then have $y_1, y_2, y_3 \in \genericworld$ such that $\mentail{\mathcal{M}}{y_i}{p_i}$ and $x \genericrelation y_i$.
It is clear by transitivity and density that if $\mentail{\mathcal{M}}{y_1}{p_2 \lor \poss p_2}$, then $\mentail{\mathcal{M}}{x}{\poss(\poss p_1 \land \poss p_2)}$ (similarly if $\mentail{\mathcal{M}}{y_2}{p_1 \lor \poss p_1}$).

If neither of the entailments hold then we must have $\mentail{\mathcal{M}}{y_i}{\lnot p_j \land \necc \lnot p_j}$ for $i, j \in \{1,2\}, i\neq j$.
Therefore, we must have that $\mentail{\mathcal{M}}{x}{\poss(p_1 \land \lnot p_2 \land \necc \lnot p_2) \land \poss(p_2 \land \lnot p_1 \land \necc \lnot p_1) \land \poss p_3}$.
By $\aftertwoformula$, we then have that $\mentail{\mathcal{M}}{x}{\poss(\poss p_1 \land \poss p_3) \lor \poss(\poss p_2 \land \poss p_3)}$.
Thus, generally, we have that $\mentail{\mathcal{M}}{x}{\poss(\poss p_1 \land \poss p_2) \lor \poss(\poss p_1 \land \poss p_3) \lor \poss(\poss p_2 \land \poss p_3)}$.
As $x$ and $\mathcal{M}$ are general we have that $\mathbf{K4} + ad + \aftertwoformula \vdash ad_{3,2}$.

\ref{thm:aftertwod32:notaftertwo}: a model depicted in Figure~\ref{fig:d32notafter2} acts as a counterexample.
One of $p_1$ or $p_2$ should share a non-root common point with $q$ but they do not.
\begin{figure}[t]
    \centering
    \begin{tikzpicture}[nodes={draw, circle}, ->,]
    \node (s) {} [grow'=up]
    child {node[refl] (1) [] {}
        child {node[refl,label=left:{$p_1$}] (10) [] {}}
        child {node[refl,label=left:{$p_2$}] (11) [] {}}
    }
    child {node[refl,label=right:{$q$}] (0) [] {}
    }
    ;
    \end{tikzpicture}
    \caption{A finite $\mathbf{K4} + ad_{3,2}$ model that does not have $\aftertwoformula$ on the root.}
    \label{fig:d32notafter2}
\end{figure}
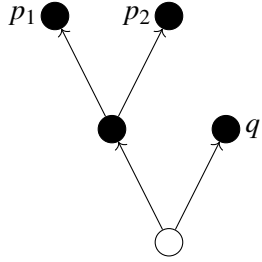

\ref{thm:aftertwod32:afterd32}: let $\mathcal{M} = \mmodel{\genericworld}{\genericrelation}{\genericeval}$ be a $\mathbf{K4} + \afterformula + ad_{3,2}$ model and $x \in \genericworld$ such that $\mentail{\mathcal{M}}{x}{\poss(p_1 \land \lnot p_2 \land \necc \lnot p_2) \land \poss(p_2 \land \lnot p_1 \land \necc \lnot p_1) \land \poss q}$.
Firstly, by $ad_{3,2}$ we have that
\begin{equation*}
\begin{aligned}
    \mathcal{M}, x  \models &
    \poss (\poss(p_1 \land \lnot p_2 \land \necc \lnot p_2) \land \poss (p_2 \land \lnot p_1 \land \necc \lnot p_1)) \lor \\ &
    \poss (\poss(p_1 \land \lnot p_2 \land \necc \lnot p_2) \land \poss q) \lor \\ &
    \poss (\poss(p_2 \land \lnot p_1 \land \necc \lnot p_1) \land \poss q).
\end{aligned}
\end{equation*}

If $\mentail{\mathcal{M}}{x}{\poss (\poss(p_1 \land \lnot p_2 \land \necc \lnot p_2) \land \poss q)}$, then $\mentail{\mathcal{M}}{x}{\poss (\poss p_1 \land \poss q)}$ and, therefore, $\mentail{\mathcal{M}}{x}{\poss (\poss p_1 \land \poss q) \lor \poss (\poss p_2 \land \poss q)}$.
The case is similar for $\mentail{\mathcal{M}}{x}{\poss (\poss(p_2 \land \dots) \land \poss q)}$.

Suppose instead $\mentail{\mathcal{M}}{x}{\poss (\poss(p_1 \land \lnot p_2 \land \necc \lnot p_2) \land \poss (p_2 \land \lnot p_1 \land \necc \lnot p_1))}$.
Then as $\mentail{\mathcal{M}}{x}{\poss q}$, by $\afterformula$, we get that $\mentail{\mathcal{M}}{x}{\poss (\poss p_1 \land \poss q) \lor \poss (\poss p_2 \land \poss q)}$.

Thus, we have $\mentail{\mathcal{M}}{x}{\aftertwoformula}$.
Since $x$ and $\mathcal{M}$ are general then $\mathbf{K4} + \afterformula + ad_{3,2} \vdash \aftertwoformula$.
\end{proof}

This gives a result on the normal modal logics of the formulae.
\begin{corollary}
    $\mathbf{D4da_2} = \mathbf{D4da} + ad_{3,2}$.
\end{corollary}

This means that when working in 2-dimensional spacetime, we have a choice of axioms to work from.
One can either work with 3,2-density and $\afterformula$, or one can work with standard density and $\aftertwoformula$ instead.

\begin{remark}
We note that $\aftertwoformula$ also bears a similarity to the modal version of $\text{Rob}^2$ in~\cite{Phillips1998}, defined by
\begin{equation*}
\begin{aligned}
    & Q_i := p_i \Rightarrow \bigwedge_{\substack{1 \leq j \leq 3\\j \neq i}} \lnot p_j \land \necc \lnot p_j, \\
    & \text{Rob}^2 :=
    \poss p_1 \land \poss p_2 \land \poss p_3 \land
    \necc Q_1 \land \necc Q_2 \land \necc Q_3
    \Rightarrow \\ & \qquad
    \poss (\poss p_1 \land \poss p_2) \lor \poss (\poss p_1 \land \poss p_3) \lor \poss (\poss p_2 \land \poss p_3).
\end{aligned}
\end{equation*}
% There are two changes to obtain $\aftertwoformula$.
% The first is that the requirement for $q$ ($p_3$) to be unrelated to $p_1$ and $p_2$ ($\necc Q_3$) is not captured in the left hand side (of the implication).
% The second is that $\poss (\poss p_1 \land \poss p_2)$ is not a term in the right hand side.
We leave whether $(\mathbf{K4} +)~\aftertwoformula \vdash \text{Rob}^2$ (or vice versa) as an open problem, but such an entailment would not be surprising.
This is because the generalisation of $\text{Rob}^2$ ($\text{Rob}^n$) captures the idea of having $n$ axes, whereas a generalisation of the $\aftertwoformula$ ($a\alpha_{n}f$) would capture the idea of an interior where any point inside can be reached from any of the $n$ border lines.
It so happens that in the $n=2$ case, these properties are equivalent, but in general they may not be.
These ideas were reflected on in Remark~\ref{rmk:diffafter} and the resulting relations depicted in Figure~\ref{fig:diffafter}.
\end{remark}

We end this section by forming the following conjecture.
\begin{conjecture}
    Let $n \geq 3$.
    For any cNTV $n$-dimensional spacetime, $M$, we have that
    $\nmentail{\mframe{M}{\after}}{}{\aftertwoformula}$ ($\aftertwologic \nsubseteq \modal{\mframe{M}{\after}}$).
\end{conjecture}

\section{Effects of Causal Ladder Properties on Modal Logics}
\label{sec:modalcausalladder}
We now analyse the modal logic of spacetimes on different levels of the causal ladder.
The proofs are fairly obvious throughout, simply relying on the properties of the relevant causal relation.

We have the following obvious result.
\begin{corollary}
    Let $\genericworld$ be a spacetime and $\chron$ and $\after$ be the respective causal relations, then:
    \begin{enumerate}[label=(\roman*)]
        \item if $\chron$ is reflexive, $\mathbf{S4} \subseteq \modal{\mframe{\genericworld}{\chron}}$;
        \item if $\after$ is reflexive, $\mathbf{S4} \subseteq \modal{\mframe{\genericworld}{\after}}$.
    \end{enumerate}
\end{corollary}

\subsection{Totally and Non-totally Vicious}
Totally vicious spacetimes are reflexive in $\chron$ and, consequently, in $\after$.
Note that $\chron$ and $\chroneq$ become indistinguishable from one another ($x \chron y$ iff $x \chroneq y$), and similarly for $\after$ and $\caus$.
What becomes more apparent is that $\chron$ and $\caus$ are indistinguishable.
\begin{lemma}
    Let $\genericworld$ be a totally vicious spacetime and $x, y \in \genericworld$.
    Then $x \chron y$ iff $x \caus y$.
\end{lemma}
\begin{proof}
    If $x \chron y$, there is a timelike path from $x$ to $y$, which is also a causal path.
    Therefore, $x \caus y$.
    If $x \caus y$, then we have that $x \chron y$ since $y \chron y$ and the push-up rule is used to obtain the result.
\end{proof}
This leads to the following result.
\begin{theorem}
    Let $\genericworld$ be a totally vicious spacetime.
    Then,
    \begin{equation*}
        \mathbf{S4} \subseteq \modal{\mframe{\genericworld}{\chron}} = \modal{\mframe{\genericworld}{\chroneq}} = \modal{\mframe{\genericworld}{\after}} = \modal{\mframe{\genericworld}{\caus}}.
    \end{equation*}
\end{theorem}

Once we move onto the causal ladder by considering non-totally vicious spacetimes, the logics that can be considered for $\chron$ are restricted since reflexivity is lost.
\begin{theorem}
    If $\genericworld$ is a non-totally vicious spacetime, then $\mathbf{S4} \nsubseteq \modal{\mframe{\genericworld}{\chron}}$.
\end{theorem}

Note that $\after$ and $\caus$ are not affected by a spacetime being NTV ($\after$ is still reflexive), and therefore their modal logics are not affected either.
\begin{corollary}
    If $M$ is NTV and not cNTV, then $\mathbf{S4} \subseteq \modal{\mframe{\genericworld}{\after}} = \modal{\mframe{\genericworld}{\caus}}$.
\end{corollary}

\subsection{Chronological}
Irreflexivity of $\chron$ means that the normal modal logic does not fundamentally change from a non-totally vicious spacetime.
However, by considering $\gabb$, we obtain the following.

\begin{theorem}
    If $M$ is a chronological spacetime, then $\mathbf{OI} + \gabb \subseteq \modal{\mframe{\genericworld}{\chron}}$.
\end{theorem}

Note how the logic of $\after$ or the reflexive causal relations have not changed at the levels of non-totally vicious and chronological from a totally vicious spacetime.

\subsection{Causally Non-Totally Vicious (cNTV)}
Having introduced the notion of cNTV in Section~\ref{sec:spacetime:cntv}, we now see the results of separating the causal condition in the modal logic of $\after$.
\begin{theorem}
    If $\genericworld$ is a causally non-totally vicious spacetime, then $\mathbf{S4} \nsubseteq \modal{\mframe{\genericworld}{\after}}$.
\end{theorem}

When considering a spacetime that is both chronological and cNTV, the two respective theorems are used together to reveal that both $\chron$ and $\after$ have different modal logics from a non-totally vicious spacetime.

\subsection{Causal}
A causal spacetime, in a similar way to a chronological spacetime, does not change the normal modal logics of its causal relations, but the logic under $\after$ can be changed if $\gabb$ is considered.
\begin{theorem}
    If $\genericworld$ is a causal spacetime, then $\afterlogic + \gabb \subseteq \modal{\mframe{\genericworld}{\after}}$.
\end{theorem}

\subsection{Distinguishing}
We begin first by extending the notion of a spacetime being distinguishing into the Kripke frame setting.

\begin{definition}
    Let $F = \mframe{\genericworld}{\genericrelation}$.
    We say $F$ is future-distinguishing if for all $x,y \in \genericworld$ whenever ${\genericrelation(x) = {\genericrelation}(y)}$, then $x = y$.
    We say $F$ is past-distinguishing if for all $x,y \in \genericworld$ whenever $\genericrelation^{-1}(x) =~{\genericrelation^{-1}(y)}$, then $x = y$.
    Finally, $F$ is distinguishing if it is both past- and future-distinguishing.
\end{definition}

Thus, the spacetime definition of distinguishing is preserved when considering the frame of a spacetime with $\chron$ (since $I^{\pm}(x) = {\genericrelation}^{\pm 1}(x)$).
Our goal is to find some normal modal logic that can capture distinguishing spacetimes.
However, despite being able to maintain definitions across domains, it becomes clear that a past-distinguishing frame has no associated modal formula.
\begin{theoremrep}
There is no modal formula, $\phi$, such that if $\mframe{\genericworld}{\genericrelation} \models \phi$ then $\mframe{\genericworld}{\genericrelation}$ is past-distinguishing.
\end{theoremrep}
\begin{proof}
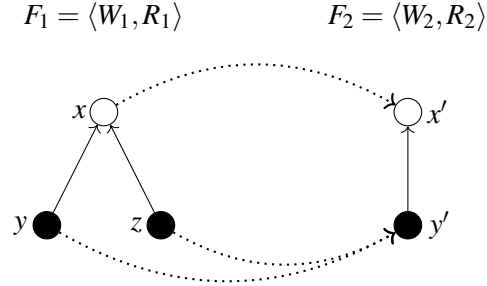
\begin{figure}[t]
    \centering
    \begin{tikzpicture}[nodes={draw, circle}, <-,]
    \node[label=above:{$F_1 = \mframe{W_1}{R_1}$},label=left:{$x$}] (s) {} [grow'=down]
        child {node[refl] (0) [label=left:{$z$}] {}}
        child {node[refl] (1) [label=left:{$y$}] {}};
    \node[label=above:{$F_2 = \mframe{W_2}{R_2}$},label=right:{$x'$}] (n) at (4,0) {} [grow'=down]
            child {node[refl] (n0) [label=right:{$y'$}] {}};

    \draw[->,dotted,thick] (s) to[bend left] (n);
    \draw[->,dotted,thick] (0) to[bend right] (n0);
    \draw[->,dotted,thick] (1) to[bend right] (n0);
    \end{tikzpicture}
    \caption{A bisimulation from a (past-)distinguishing frame to one that is not.}
    \label{fig:pastdistframes}
\end{figure}
The frames in Figure~\ref{fig:pastdistframes} act as a counterexample.
The left frame, $F_1$, is past-distinguishing ($F_1 \models \phi$); but the right frame, $F_2$, is not ($F_2 \models \lnot \phi$).
There is a bisimulation, $Z \subset F_1 \times F_2$, depicted by the dotted lines ($Z = \{(x, x'), (y, y'), (z,y')\}$).
Since modal formulae are invariant under bisimulation, this implies that $F_2 \models \phi$, which is a contradiction.\footnote{See Section~\ref{sec:bisim}/\cite{Blackburn2001} for information on bisimulation.}
\end{proof}

Similarly, future-distinguishing frames have no associated modal formula.
\begin{theoremrep}
    There is no modal formula, $\phi$, such that if $\mframe{\genericworld}{\genericrelation} \models \phi$ then $\mframe{\genericworld}{\genericrelation}$ is future-distinguishing.
\end{theoremrep}
\begin{proof}
A contradiction is obtained by reversing the relation in Figure~\ref{fig:pastdistframes} ($F_1' = \mframe{W_1}{R_1^{-1}}$, $F_2' = \mframe{W_2}{R_2^{-1}}$) and observing a similar bisimulation from $F_1'$ to $F_2'$.
\end{proof}

\begin{corollaryrep}
    A frame being distinguishing is not modally definable.
\end{corollaryrep}
\begin{proof}
    We have that $F_1$ is (past- and future-) distinguishing, but $F_2$ is not (past-) distinguishing.
\end{proof}

We note that whilst it is infeasible to use a modal formula to separate distinguishing frames from non-distinguishing frames, it might be feasible to use some rule that characterises distinguishing frames in a similar way as to how irreflexive frames are characterised by a rule~\cite{Gabbay1981}.
Such an investigation is beyond the focus of this work.

% Globally hyperbolic not modally definable but don't know how :(

\section{Discussion}
% \subsection*{Modal Logics of the Causal Ladder}
\paragraph{Modal logics of the causal ladder.}
\begin{table}[t]
    \centering
    \caption{The modal logics of spacetime at different levels of the causal ladder.
    Blank cells use the same logic as the cell above.}
    \begin{tabular}{|c|c|c|c|}
         \cline{2-4}
         \multicolumn{1}{}{} & \multicolumn{3}{|c|}{Spacetime Modal Logic Containment ($L = \modal{\mframe{M}{\genericrelation}}$)} \\ \hline
         Spacetime Type ($M$) & $\chron$& $\after$ & $\chroneq~/\caus$ \\ \hline
         Totally Vicious & $\mathbf{S4} \subseteq L$ & $\mathbf{S4} \subseteq L$ & $\mathbf{S4} \subseteq L$ \\ \hline
         Non-Totally Vicious & $\mathbf{OI} \subseteq L$ & & \\ \hline
         Chronological & $\mathbf{OI}+\gabb \subseteq L$ & & \\ \hline
         Causally Non-Totally Vicious & $\mathbf{OI}$ & $\afterlogic \subseteq L $ & \\ \hline
         cNTV and Chronological & $\mathbf{OI}+\gabb \subseteq L$ & & \\ \hline
         Casual & & $\afterlogic+\gabb \subseteq L$ & \\ \hline
         Distinguishing & & & \\ \hline
         $\vdots$ & \multicolumn{3}{|c|}{$\vdots$} \\ \hline
         Globally hyperbolic & $\mathbf{OI}+\gabb \subseteq L$ & $\afterlogic +\gabb \subseteq L$ & $\mathbf{S4} \subseteq L$ \\ \hline
    \end{tabular}
    \label{tab:modalladder}
\end{table}

Table~\ref{tab:modalladder} shows the containment of the various modal logics for spacetimes and their various causal relations on different levels of the causal ladder.
We note that spacetimes beyond the cNTV condition, experience no changes in their normal modal logic.
Even allowing Gabbay's irreflexive rule, there is no change in the simplest modal logic after the causal level of the ladder.

The later levels of the causal ladder require new rules or additional operations to distinguish between their logics on the ladder.
Note also how $\chroneq$ and $\caus$ do not change their base normal modal logic on any level of the causal ladder.
The logics may change if an additional operation is added that provides additional formula or rules due to, \eg{} antisymmetry, which occur for $\chroneq$ and $\caus$ at the chronological and causal levels respectively.

As explored in Section~\ref{sec:separable}, two-dimensional spacetimes have a more expressive modal logic than generic spacetimes.
Thus for two-dimensional spacetimes, $\afterlogic$ is replaced with $\aftertwologic$ in Table~\ref{tab:modalladder}.

% \subsection*{Other Effects on the Logic}
\paragraph{Other effects on the logic.}
Note that spacetimes on the same rung of the causal ladder may nevertheless exhibit distinct modal logics.
For instance, contrast ordinary two-dimensional Minkowski space with the $2^{-}$-Minkowski spacetime: $\{(t,x)\in\mathbb{R}^2: t <0\}$, inheriting the same metric/causal structure as in Example~\ref{example:minkowski space}.
Both are globally hyperbolic, but $2^{-}$-Minkowski spacetime does not have the (future-directed) confluence property for the causal relation (see~Figure~\ref{fig:diamond}), therefore exhibiting different modal logics.
The logics of the spacetimes under the causal (chronological) relation are $\mathbf{S4.2}$~\cite{Goldblatt1980} ($\mathbf{OI.2}$~\cite{Shapirovsky2002}) and $\mathbf{S4}$ ($\mathbf{OI}$) \cite[Lemma 3.6]{Shapirovsky2005} respectively.

Conversely, it is clear that the modal structure is not expressive enough to completely characterise the spacetime, since all Minkowski spaces of total dimension greater than three share the same modal logic. More can perhaps be said using a multi-modal logic, where it is possible to quantify over causal and chronological curves, and then employing Malament's theorem~\cite{malament1977ClassContinuousTimelike}.

\begin{figure}[t]
    \centering
    \begin{subfigure}{.39\textwidth}
        \begin{tikzpicture}[scale=1.4]
    \tikzminkowski
    \node[label={[label distance=.05cm]-45:$w$}] (w) at (0,-1) {$\bullet$};
    \node[label={[label distance=.05cm]180:$x$}] (x) at (-.75,-.25) {};
    \node[label={[label distance=.05cm]0:$y$}] (y) at (.75,-.25) {};
    \node[label={[label distance=.05cm]45:$z$}] (z) at (0,.5) {$\bullet$};
    \draw (w.center) -- (x.center);
    \draw (w.center) -- (y.center);
    \draw (x.center) -- (z.center);
    \draw (y.center) -- (z.center);
\end{tikzpicture}
        \caption{Minkowski spacetime}
    \end{subfigure}
    \hfill
    \begin{subfigure}{.39\textwidth}
        \begin{tikzpicture}[scale=1.4]
    \tikzminkowski
    \fill [pattern={dots}] (-\xmax,0) -- (-\xmax,\xmax) -- (\xmax+0.2,\xmax) -- (\xmax+0.2,0);
    \node[label={[label distance=.05cm]-45:$w$}] (w) at (0,-1) {$\bullet$};
    \node[label={[label distance=.05cm]180:$x$}] (x) at (-.75,-.25) {};
    \node[label={[label distance=.05cm]0:$y$}] (y) at (.75,-.25) {};
    \node (xl) at (-1,0) {};
    \node (xr) at (-.5,0) {};
    \node (yl) at (.5,0) {};
    \node (yr) at (1,0) {};
    \draw (w.center) -- (x.center);
    \draw (w.center) -- (y.center);
    \draw (x.center) -- (xl.center);
    \draw (x.center) -- (xr.center);
    \draw (y.center) -- (yl.center);
    \draw (y.center) -- (yr.center);
    \fill [pattern={Lines[angle=45,distance=3pt]}] (x.center) -- (xl.center) -- (xr.center);
    \fill [pattern={Lines[angle=45,distance=3pt]}] (y.center) -- (yl.center) -- (yr.center);
\end{tikzpicture}
        \caption{$2^{-}$-Minkowski spacetime}
    \end{subfigure}
    \caption{Confluence on different spacetimes}
    \label{fig:diamond}
\end{figure}
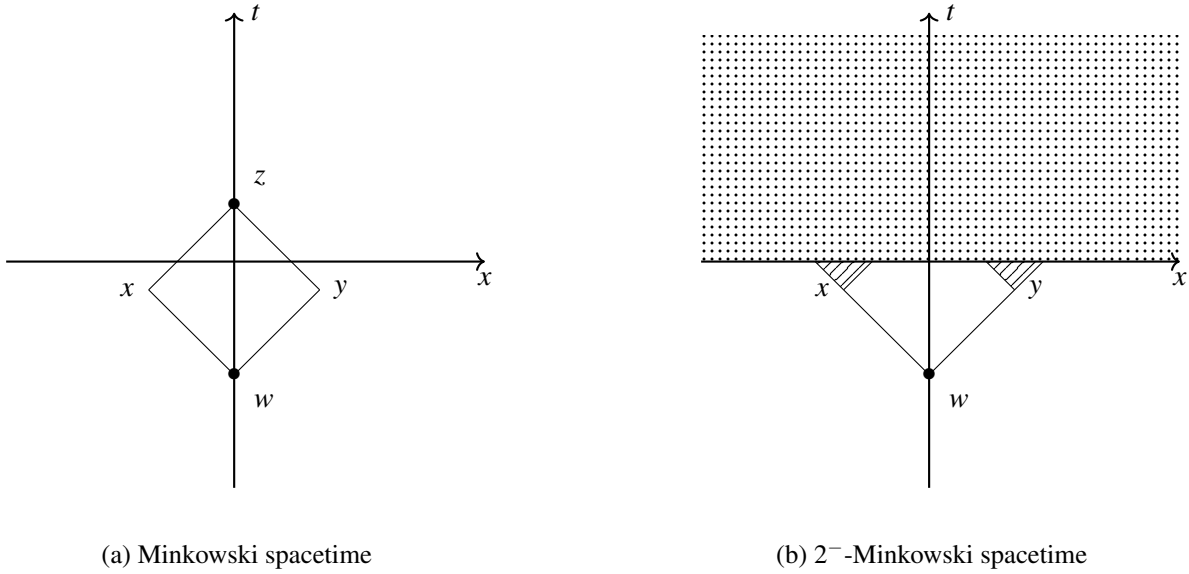

% \subsection*{Multiple Time Dimensions}
\paragraph{Multiple time dimensions.}
Whilst we have discussed the ideas behind spacetimes with arbitrary spatial dimension, there is a notion of spacetimes with multiple time dimensions.
A $n$-spatial, $t$-temporal dimensional (or $(n,t)$-dimensional) spacetime can then be described as working on a $n+t$-dimensional manifold.

Our universe, $(3,1)$-dimensional, can be considered ``nice''.
By changing the number of spatial or temporal dimensions, the resulting universe could be considered either too simple or too unstable (not ``nice'').
For instance, Tegmark~\cite{Tegmark1997} provides various arguments for why observers cannot exist in $(n,t)$-dimensional universes where $(n, t) \neq (3,1)$, but the arguments made are not rigorous.

One fundamental question to ask is whether there is a fundamental logical reason that these universes are not ``nice''.
Whilst it is predicted these other universes may not have observers, can this (or some other property) be rigorously shown or proven through logical reasoning of the causal structure.
In Section~\ref{sec:separable:separability}, we showed that modal logic seems expressible enough to distinguish between $(1,1)$-dimensional and $(n,1)$-dimensional spacetimes ($n \geq 2$), but it is unlikely to distinguish any higher (spatial) dimensional spacetime from others.
Extending our definition of $\necc$ and $\poss$ on the causal relations to higher temporal dimensions, or using multimodal logic to represent causal relations in different time dimensions, could provide the start towards a logical language to discuss what universes are ``nice'' or not.

\shelf{Calabi-Yau Manifold}

\section{Conclusion}
In this work, we have explored modal logic in the frames of general smooth spacetimes and their various causal relations ($\chron$, $\chroneq$, $\after$, $\caus$).
We have shown the minimal normal modal logics of general spacetime and have explored the after formula, $\afterformula$, within the modal logics of spacetimes, showing that it exists in the logic for each causal relation.
Further, we have shown that there is some separation between the normal modal logic of the after relation, $\after$, with two-dimensional spacetimes, which contains the after-2 formula, $\aftertwoformula$, and 3,2-density, $ad_{3,2}$.
Finally, we have seen how the logics of a spacetime changes depending on its position on the causal ladder.

We have already explored several avenues for future research.
Whilst we can already see that modal logic is capable of distinguishing some spacetimes from others, the more complex levels of the causal ladder require a more expressive temporal logic to express the relative behaviours.
What sort of operations are needed to express these properties remains an open question.

\section*{Acknowledgements}
\change{We thank Prakash Panangaden for helpful comments and feedback that helped improve our exposition, and for spotting a gap in the proof of~Section~\ref{sec:spacetime:afterformula}.}
This work has been partially funded by the French National Research Agency (ANR) within the framework of ``Plan France 2030'', under the research projects EPIQ ANR-22-PETQ-0007, HQI-Acquisition ANR-22-PNCQ-0001 and HQI-R\&D ANR-22-PNCQ-0002.

% \nocite{*}
\bibliographystyle{eptcs}
\bibliography{bibliography}

\end{document}